\documentclass[reqno, a4paper]{amsart}

\usepackage[T1]{fontenc}
\usepackage[utf8]{inputenc}
\usepackage[british]{babel}
\usepackage{graphicx}
\usepackage{amssymb}
\usepackage{amsthm}
\usepackage{hyperref}
\usepackage{mathbbol}
\usepackage{mathabx}
\usepackage{multicol}
\usepackage{calrsfs}
\usepackage[all]{xy}

\theoremstyle{plain} 
\newtheorem{theorem}[subsection]{Theorem}
\newtheorem{proposition}[subsection]{Proposition}
\newtheorem{corollary}[subsection]{Corollary}
\newtheorem{lemma}[subsection]{Lemma}

\theoremstyle{definition}
\newtheorem{definition}[subsection]{Definition}

\theoremstyle{remark}
\newtheorem{remark}[subsection]{Remark}
\newtheorem{example}[subsection]{Example}

\newtheorem{exercise}[subsection]{Exercise}

\newenvironment{ppmatrix}
{\left\lgroup\begin{matrix}}
{\end{matrix}\right\rgroup}

\newcommand{\defn}[1]{{\bf #1}}

\newcommand{\C}{\ensuremath{\mathcal{C}}}
\newcommand{\D}{\ensuremath{\mathcal{D}}}
\newcommand{\E}{\ensuremath{\mathcal{E}}}
\newcommand{\V}{\ensuremath{\mathcal{V}}}

\newcommand{\J}{\ensuremath{\mathfrak{J}}}

\newcommand{\K}{\ensuremath{\mathbb{K}}}
\newcommand{\N}{\ensuremath{\mathbb{N}}}

\newcommand{\R}{\ensuremath{\mathbb{R}}}
\newcommand{\Z}{\ensuremath{\mathbb{Z}}}

\newcommand{\AbAlg}{\ensuremath{\mathsf{AbAlg}}}

\newcommand{\Lie}{\ensuremath{\mathsf{Lie}}}
\newcommand{\Liek}{\ensuremath{\mathsf{Lie}}_{\K}}

\newcommand{\Alg}{\ensuremath{\mathsf{Alg}}}
\newcommand{\AssocAlg}{\ensuremath{\mathsf{AssocAlg}}}

\newcommand{\qLiek}{\ensuremath{\mathsf{qLie}}_{\K}}

\newcommand{\Vect}{\ensuremath{\mathsf{Vect}}}

\newcommand{\Mag}{\ensuremath{\mathsf{Mag}}}
\newcommand{\Set}{\ensuremath{\mathsf{Set}}}

\newcommand{\kar}{\ensuremath{\mathrm{char}}}
\newcommand{\Hom}{\ensuremath{\mathrm{Hom}}}
\newcommand{\Der}{\ensuremath{\mathrm{Der}}}
\newcommand{\End}{\ensuremath{\mathrm{End}}}
\newcommand{\Forget}{\ensuremath{\mathrm{Forget}}}

\newcommand{\Free}[1]{\ensuremath{\K\ldbrack{#1}\rdbrack}}

\newcommand{\LACC}{{\rm (LACC)}}

\newdir{>>}{{}*!/3.5pt/:(1,-.2)@^{>}*!/3.5pt/:(1,+.2)@_{>}*!/7pt/:(1,-.2)@^{>}*!/7pt/:(1,+.2)@_{>}}
\newdir{ >>}{{}*!/8pt/@{|}*!/3.5pt/:(1,-.2)@^{>}*!/3.5pt/:(1,+.2)@_{>}}
\newdir{ |>}{{}*!/-3.5pt/@{|}*!/-8pt/:(1,-.2)@^{>}*!/-8pt/:(1,+.2)@_{>}}
\newdir{ >}{{}*!/-8pt/@{>}}
\newdir{>}{{}*:(1,-.2)@^{>}*:(1,+.2)@_{>}}
\newdir{<}{{}*:(1,+.2)@^{<}*:(1,-.2)@_{<}}

\begin{document}

\author[Summer School in Algebra and Topology]{Tim Van~der Linden}
\address[Tim Van~der Linden]{Institut de
Recherche en Math\'ematique et Physique, Universit\'e catholique
de Louvain, che\-min du cyclotron~2 bte~L7.01.02, B--1348
Louvain-la-Neuve, Belgium}
\email{tim.vanderlinden@uclouvain.be}

\thanks{These lecture notes were prepared for the \emph{Summer School in Algebra and Topology} held at the \emph{Institut de Recherche en Mathématique et Physique} of the \emph{Université catholique de Louvain}, 12th--15th September 2018. The current, revised version dates \today.}
\thanks{Tim Van der Linden is a Research Associate of the Fonds de la Recherche Scientifique--FNRS}


\begin{abstract}
A \emph{non-associative algebra} over a field $\K$ is a $\K$-vector space $A$ equipped with a bilinear operation
\[
{A\times A\to A\colon\; (x,y)\mapsto x\cdot y=xy}.
\]
The collection of all non-associative algebras over $\K$, together with the product-preserving linear maps between them, forms a variety of algebras: the category $\Alg_\K$. The multiplication need not satisfy any additional properties, such as associativity or the existence of a unit. Familiar categories such as the varieties of associative algebras, Lie algebras, etc.\ may be found as subvarieties of~$\Alg_\K$ by imposing equations, here $x(yz)=(xy)z$ (associativity) or $xy =- yx$ and $x(yz)+z(xy)+ y(zx)=0$ (anti-commutativity and the Jacobi identity), respectively.
\newline\indent
The aim of these lectures is to explain some basic notions of categorical algebra from the point of view of non-associative algebras, and vice versa. As a rule, the presence of the vector space structure makes things easier to understand here than in other, less richly structured categories.
\newline\indent
We explore concepts like normal subobjects and quotients, coproducts and protomodularity. On the other hand, we discuss the role of (non-associative) polynomials, homogeneous equations, and how additional equations lead to reflective subcategories.	
\end{abstract}

\title{Non-associative algebras}

\maketitle

\tableofcontents

\section{Introduction}

An \emph{algebra} is a vector space over a field, equipped with an additional bilinear multiplication. In practice, in the literature, and historically, this algebra multiplication is usually taken to be associative, often also commutative, or else it is asked to satisfy some other rule like the Jacobi identity. For us here, such special types of algebras will be important classes of examples. However, in these lectures about \emph{non-associative algebras} we will \emph{a priori} not ask that the multiplication of an algebra satisfies any rule at all, besides that it is bilinear. This makes it possible to treat many different types of algebras all at once, and interpret their common properties by means of categorical-algebraic concepts.

From the point of view of Category Theory, vector spaces are simple and extremely well behaved. (Not only is a category of vector spaces always abelian; also the fact that every vector space has a basis is important for us.) As we shall see, adding a multiplication to a vector space actually makes its behaviour worse---so for us, more interesting. The reason is, that this allows us to study categorical concepts which are often trivial for vector spaces, but whose definition makes full sense in a context which is only slightly different. In a way, it is still all Linear Algebra. The categorical properties which we shall treat here are mainly those related to the context of semi-abelian categories, which find a concrete use for instance in homology of non-abelian objects.	

\section{Non-associative algebras}
Throughout these lectures, we fix a field $\K$. 

\begin{definition}\label{Def:NAA}
A \defn{(non-associative) algebra $(A,\cdot)$ over $\K$} is a $\K$-vector space~$A$ equipped with a bilinear operation ${\cdot\colon A\times A\to A\colon (x,y)\mapsto x\cdot y}$.	
\end{definition}

Recall that $\cdot$ being \defn{bilinear} means that it is linear in both variables $x$ and $y$, so that for all $\lambda\in \K$ and $x$, $x'$, $y$, $y'\in A$,
\begin{gather*}
(x+x')\cdot y=x\cdot y+ x'\cdot y\text{,}
\qquad\qquad
x\cdot (y+y')=x\cdot y+ x\cdot y'\text{,}\\
(\lambda x)\cdot y=\lambda (x\cdot y)=x\cdot (\lambda y)
\text{.}
\end{gather*}
In other words, it induces a unique linear map $A\otimes A\to A$ sending elements of the form $x\otimes y\in A\otimes A$ to $x\cdot y\in A$.

Depending on the type of non-associative algebra which is being considered, to write the multiplication, several notations are common. We shall often drop the dot and write $xy$ for the product $x\cdot y$. In a context related to Lie algebras, $x\cdot y$ is usually written as a bracket $[x,y]$. 

Unless when this would be confusing, we write $A$ for an algebra $(A,\cdot)$, dropping the multiplication $\cdot$.

\section{Examples}

In practice, additional conditions are imposed on the multiplication, and then only those algebras that satisfy these conditions are considered. For instance, we may ask that the multiplication of an algebra $A$ is commutative, which means that $xy=yx$ for all elements $x$, $y$ of $A$, and decide to study only such algebras. 

The word ``condition'' here means any set of equations that the multiplication on an algebra should satisfy. In this section we give a number of examples of such conditions. Later we make the concept of ``a condition'' itself more precise and investigate it in general. 

\begin{definition}
We write $\Alg_{\K}$ for the category of non-associative algebras over~$\K$, with linear maps $f\colon {A\to B}$ that preserve the multiplication:
\[
f(\lambda x)=\lambda f(x)
\text{,\quad}
f(x+y)=f(x)+f(y)
\text{\quad and \quad}
f(x\cdot y)=f(x)\cdot f(y)
\]
for $\lambda\in \K$ and $x$, $y\in A$. 
\end{definition}

Note that the identity function $1_A\colon {A\to A\colon x\mapsto x}$ on an algebra $A$ is always an algebra map, and such is the composite $g\circ f\colon {A\to C}$ of two algebra maps $f\colon {A\to B}$ and $g\colon{B\to C}$.

\begin{definition}
Given categories $\C$ and $\D$, if the objects and arrows of $\D$ are also objects and arrows in $\C$, and if the identities and the composition in $\D$ agree with those in $\C$, then $\D$ is called a \defn{subcategory} of $\C$. 

A subcategory $\D$ of $\C$ is a \defn{full subcategory} when the arrows in $\D$ are precisely the arrows in $\C$ between objects in $\D$: for each pair of objects $X$, $Y$ of $\D$, $\Hom_\D(X,Y)=\Hom_\C(X,Y)$.
\end{definition}

Note that a full subcategory of a given category $\C$ is completely determined by a choice of objects in $\C$. 

\begin{definition}\label{Def:Variety}
The collection of all $\K$-algebras satisfying a chosen equational condition is called a \defn{variety of non-associative algebras}. It is any class of algebras defined by a (possibly infinite) set of equations, considered as a full subcategory~$\V$ of~$\Alg_{\K}$. An object of $\V$ is often called a \defn{$\V$-algebra}.
\end{definition}

In other words, the morphisms in a variety of non-associative algebras are again the linear maps preserving the multiplication. In Sections~\ref{S8} and~\ref{S12} we will explain what we mean by ``equational condition'' in Definition~\ref{Def:Variety}. Essentially it will consist of a possibly infinite system of equations $\psi_\lambda=0$, $\lambda\in \Lambda$, where each $\psi_\lambda$ is a non-associative polynomial.

\begin{example}
An algebra is said to be \defn{commutative} if $xy=yx$ for all $x$ and $y$.
\end{example}

\begin{example}\label{Ex:Associative}
An algebra $A$ is \defn{associative} when $x(yz)=(xy)z$ for all elements $x$, $y$, $z$ of $A$. All associative algebras together form a variety of non-associative algebras, $\AssocAlg_{\K}$: so in these lectures, ``non-associative'' means ``not necessarily associative''.
\end{example}

\begin{remark}
It is easy to check that an associative $\Z$-algebra is the same thing as a ring (with or without unit). This doesn't fit our definition of a non-associative algebra though, because $\Z$ is not a field. We could have made Definition~\ref{Def:NAA} more general, so that this example could be included. We chose not to do this, because it would make other aspects of the theory significantly more complicated.
\end{remark}

\begin{exercise}
However, a lot can still be done. Throughout the text, try to extend the theory from vector spaces to arbitrary rings, and discover where this may make things more difficult.
\end{exercise}

Associativity may be weakened or modified, as in the next two examples.

\begin{example}
An \defn{alternative algebra} satisfies the equations $x(xy)=(xx)y$ and $(yx)x=y(xx)$. An \defn{anti-associative} algebra satisfies the equation $x(yz)=-(xy)z$.
\end{example}

We may also impose much stronger conditions, such as the next one.

\begin{example}[Vector spaces as non-associative algebras]\label{Ex:Abelian}
Any vector space $V$ may be considered as a 	non-associative algebra, by imposing a trivial multiplication: $xy=0$, the product of any $x$, $y\in V$ is the zero vector of $V$.

We write $\Vect_{\K}$ for the category of $\K$-vector spaces and linear maps between them. It is isomorphic to the variety $\AbAlg_{\K}$ of so-called \defn{abelian} non-associative algebras, which are those determined by the equation $xy=0$. This is of course not the same thing as the commutativity condition $xy=yx$ in Example~\ref{Ex:Associative}.

Indeed, if the functor that equips a vector space with the trivial multiplication is written $T\colon{\Vect_{\K}\to \AbAlg_{\K}}$, and $U\colon \AbAlg_{\K}\to \Vect_{\K}$ is the functor which forgets the multiplication of a trivial algebra, then clearly $T\circ U=1_{\AbAlg_{\K}}$ and $U\circ T=1_{\Vect_{\K}}$.
\end{example}

\emph{Lie algebras} (Example~\ref{Ex:Lie}) are of a wholly different nature. First of all, they are \emph{alternating}:

\begin{example}\label{Ex:Alt}
An algebra $A$ is called \defn{alternating} when $xx=0$ for every $x$ in $A$. It is said to be \defn{anti-commutative} when $xy=-yx$ for all $x$, $y\in A$. 
\end{example}

When it exists, the smallest positive integer $n$ such that
\[
n=\underbrace{1+\dots +1}_{\text{$n$ terms}}
\]
is zero in $\K$ is called the \defn{characteristic} of the field $\K$. If such a positive integer does not exist---which can only happen when $\K$ is infinite---then we say that the characteristic of $\K$ is $0$. If the characteristic of the field $\K$ is different from $2$, then these two conditions (being alternating, being anti-commutative) are equivalent. If $xx=0$ for every $x$ in $A$, then
\[
0=(a+b)(a+b)= aa+ab+ba+bb=ab+ba
\]
for all $a$, $b\in A$. So \emph{alternating implies anti-commutative}. Conversely, since we can take $x=y$, the equation $xy=-yx$ implies $xx=-xx$, hence $0=xx+xx=(1+1)xx=2xx$. So unless $0=2$ in the field~$\K$, this implies that $xx=0$. 

However, they are not equivalent in general: the simplest example of a field of characteristic $2$ is the field $\Z_2=\{0,1\}=\Z/2$ of integers modulo $2$. Over~$\Z_2$, the $2$-dimensional vector space with basis $\{x,y\}$ becomes an anti-commutative algebra which is not alternating if we define its multiplication as $xx=y$ and ${xy=yx=yy=0}$. Note that $xx=y=-y=-xx$.

\begin{example}\label{Ex:Lie}
The category $\Lie_{\K}$ of \defn{Lie algebras} over $\K$ consists of those alternating algebras satisfying the \defn{Jacobi identity}
\[
x(yz)+z(xy)+ y(zx)=0\text{.}
\] 

Lie algebras are notorious because of their connection with \emph{Lie groups}, which are smooth manifolds that carry a (compatible) group structure. Actually, each Lie group induces a Lie algebra over $\R$, and this process gives rise to a non-trivial equivalence of suitably chosen subcategories.

Another source of Lie algebras (over any field) are those coming from associative algebras. There is a functor $G\colon {\AssocAlg_\K\to \Lie_\K}$ which takes an associative algebra $(A,\cdot)$ and sends it to the couple $(A,[-,-])$ where
\[
[-,-]\colon A\times A\to A\colon (x,y)\mapsto [x,y]=xy-yx\text{.}
\]
It is easy to check (Exercise~\ref{Exer:LieStructure}) that this bracket does indeed define a Lie algebra structure on $A$. The functor $G$ sends a morphism of associative algebras to the same linear map, now a morphism of Lie algebras, since it automatically preserves the bracket.

Note that two elements $x$, $y$ of $(A,\cdot)$ commute ($xy=yx$) if and only if their bracket vanishes ($[x,y]=0$); so the associative algebra $(A,\cdot)$ is commutative if and only if the Lie algebra $(A,[-,-])$ is abelian.

The functor $G$ is not an equivalence of categories---the two types of structure are fundamentally different---but it has a left adjoint $\Lie_\K\to \AssocAlg_\K$ which is called the \emph{universal enveloping algebra} functor; see Definition~\ref{Def:Adj}.

A third example of Lie $\K$-algebra is, for any associative $\K$-algebra $A$, the \defn{Lie $\K$-algebra of derivations} $\Der_\K(A)$ of the $\K$-algebra $A$. A \defn{derivation} of a $\K$-algebra $A$ is any $\K$-linear mapping $D\colon A\to A$ such that $D(xy)=(D(x))y+x(D(y))$ for every $x$, $y\in A$. If $A$ is any associative $\K$-algebra and $D$, $D'$ are two derivations of~$A$, then $DD'-D'D$ is a derivation. Thus, for any associative $\K$-algebra $A$, it is possible define the Lie $\K$-algebra $\Der_\K(A)$ as the subset of $\End_\K(A)$ consisting of all derivations of $A$ with multiplication $[D,D']\coloneq DD'-D'D$ for every $D$, $D'\in\Der_\K(A)$.
\end{example}

\begin{exercise}
	Read about Lie groups and how they give rise to Lie algebras.
\end{exercise}

\begin{exercise}\label{Exer:LieStructure}
Show that the bracket above defines a Lie algebra structure on $A$. 
\end{exercise}

\begin{exercise}
	Investigate the universal enveloping algebra functor.
\end{exercise}

\begin{example}
Instead of being alternating, we may ask that the multiplication of an algebra satisfying the Jacobi identity is anti-commutative ($xy=-yx$). Then this algebra is called a \defn{quasi-Lie algebra}. The variety $\qLiek$ of quasi-Lie algebras coincides with $\Lie_{\K}$ as long as the characteristic of the field $\K$ is different from~$2$. However, when $\kar(\K)=2$, the variety $\Liek$ is strictly smaller than $\qLiek$: the algebra over $\Z_2$ given in Example~\ref{Ex:Alt} is a quasi-Lie algebra which is not Lie.

Approaches to Lie algebras via \emph{operads} usually deal with quasi-Lie algebras instead, because the repetition of the variable $x$ which occurs in the equation $xx=0$ cannot be expressed within that framework, so the equation $xy=-yx$ serves as a substitute.
\end{example}

We may further weaken or modify this definition as follows.

\begin{example}\label{Ex:Leibniz}
	A \defn{Leibniz algebra} is a non-associative algebra satisfying a variation on the Jacobi identity, namely $(xy)z=x(yz)+(xz)y$. It is easy to see that an anti-commutative algebra is a Leibniz algebra if and only if it is a quasi-Lie algebra.
\end{example}

\begin{example}
A \defn{Jordan algebra} is a commutative algebra ($xy=yx$) which satisfies the \defn{Jordan identity} $(xy)(xx)=x(y(xx))$. 

If a non-associative algebra is commutative and satisfies the Jacobi identity, then it is called a \defn{Jacobi--Jordan algebra}. Over a field of characteristic $2$, quasi-Lie algebras and Jacobi-Jordan algebras coincide (since commutative = anti-commutative). In particular then, they are Jordan algebras: indeed, the Jacobi identity implies that $3x(xx)=0$, so $x(xx)=0$; then via Example~\ref{Ex:Leibniz}, we see that $(xy)(xx)=x(y(xx))+(x(xx))y=x(y(xx))$.
\end{example}

Many more examples of varieties of non-associative algebras exist in the literature. We end with two extreme ones:

\begin{example}
The largest variety of non-associative $\K$-algebras is $\Alg_\K$ itself (no conditions) and the smallest one is the \defn{trivial variety} $0$ (consisting of the zero algebra only, satisfying all equations possible, including $x=0$).
\end{example}

\begin{example}
\defn{Unitary} (associative) algebras---those $(A,\cdot)$ which have an element $1$ for which $x\cdot 1=x=1\cdot x$---do not form a variety in our sense, since the \emph{existence} of $1$ cannot be expressed as an equational condition. This does not mean that an algebra cannot have a unit. On the other hand, even between algebras with units, \emph{a priori} there is no reason why a morphism of algebras should preserve this unit. 
\end{example}

The aim is now to explore some basic categorical concepts in the context of non-associative algebras. Most of what we are going to prove here may be seen as consequences of more general results, but to take that approach would defeat our purpose of keeping things simple. 

We work in a chosen variety of non-associative algebras~$\V$. The concepts we shall define do not depend on which variety we choose; here is a little example. Recall that a morphism $f\colon{A\to B}$ in a category is an \defn{isomorphism} if and only if there exists a morphism $g\colon{B\to A}$ such that $f\circ g=1_B$ and ${g\circ f=1_A}$. 

\begin{lemma}\label{Lem:Iso}
In a variety of non-associative algebras $\V$, a morphism is an isomorphism if and only if it is a bijection.
\end{lemma}
\begin{proof}
It follows immediately from the definition that any isomorphism of non-associative algebras is an isomorphism of its underlying sets, which makes it a bijection. Conversely, let $f\colon{A\to B}$ be a bijective morphism in $\V$. Then we need to show that the inverse function $g\colon{B\to A}$ is also a morphism in $\V$, i.e., a $\K$-algebra morphism. This is easy to see, using that $f$ is an injective morphism. For instance, $g(x\cdot y)=g(x)\cdot g(y)$ for any $x$, $y\in B$, because $f(g(x\cdot y))=x\cdot y=f(g(x))\cdot f(g(y))=f(g(x)\cdot g(y))$.
\end{proof}

\begin{remark}
Like several other results in these notes, this lemma is valid in the far more general context of a \emph{variety} in the sense of universal algebra. A precise definition of this concept may be found for instance in~\cite{MacLane}. A good exercise is to investigate which of our results generalise to that larger setting.
\end{remark}

\section{The zero algebra; kernels and cokernels}

Any variety of non-associative $\K$-algebras contains the zero-dimensional $\K$-vector space $0$ (whose unique element is also denoted $0$) as an object. Its multiplication is the unique map $\cdot\colon 0\times 0\to 0\colon (0,0)\mapsto 0\cdot 0=0$, which of course satisfies all possible equations. Categorically, this algebra is a \emph{zero object}, and its existence makes any variety of non-associative algebras into a \emph{pointed category}. 

\begin{definition}
An object $T$ is \defn{terminal} in a category $\C$ when for each object~$A$ of~$\C$ there is exactly one arrow $A\to T$ in $\C$. An object $I$ is \defn{initial} in $\C$ when for each object~$B$ of $\C$ there is exactly one arrow $I\to B$ in $\C$.

A \defn{zero object} or \defn{null object} in a category $\C$ is an object (denoted $0$) which is both initial and terminal. Given any two objects $A$ and $B$ of $\C$, there is a unique \defn{zero arrow} $0\colon A\to 0\to B$.
\end{definition}

When a category has a terminal (or an initial) object, this object is necessarily unique up to isomorphism: for any two terminal objects $T$ and $T'$ there are unique arrows $T\to T'$ and $T'\to T$, which are each other's inverse, because their composites ${T\to T'\to T}$ and ${T'\to T\to T'}$ are necessarily equal to the identity morphism on~$T$ and~$T'$, respectively.

\begin{example}
	In the category $\Set$ of sets and functions, the empty set is initial, and any one-element set is terminal. The same holds for topological spaces, when these sets are equipped with their unique topology. 
\end{example}

The zero algebra is a zero object in any variety of non-associative algebras, because the unique linear map to it or from it does automatically preserve the multiplication. Between any two algebras $A$ and $B$, the zero arrow is the morphism $A\to B\colon x\mapsto 0$.

\begin{proposition}
Any variety of non-associative algebras is a pointed category, where the zero object is the zero algebra.
\end{proposition}

The context of a pointed category is the right categorical environment for the definition of the concepts of the \emph{kernel} and \emph{cokernel} of a morphism.

\begin{definition}
In a pointed category $\C$, an arrow $k\colon {K\to A}$ is a \defn{kernel} of an arrow $f\colon {A\to B}$ when $f\circ k=0$, and every other arrow $h\colon {C\to A}$ such that $f\circ h=0$ factors uniquely through $k$ via a morphism $h'\colon {C\to K}$ such that $k\circ h'=h$.
\[
\xymatrix@R=1ex{K \ar[rd]^-k \\
& A \ar[r]^-f & B\\
C \ar[ru]_-{\forall h} \ar@{-->}[uu]^-{\exists !h'} }
\] 
\end{definition}

In other words, in the category where an object is an arrow $h\colon{C\to A}$ such that $f\circ h=0$ and a morphism $g\colon h\to i$ between such arrows is a commutative triangle 
\[
\xymatrix@R=1ex{D \ar[rd]^-i \\
& A \ar[r]^-f & B\text{,}\\
C \ar[ru]_-{h} \ar@{->}[uu]^-{g} }
\] 
a kernel of $f$ is a terminal object. Hence kernels are unique up to isomorphism. Because of this, when a kernel exists, we sometimes speak about ``the kernel'' (or ``the cokernel'') of a morphism.

In a variety of non-associative algebras, for any arrow $f\colon {A\to B}$ there exists a kernel $k\colon {K\to A}$, namely its kernel in the vector space sense. Explicitly, the algebra $K$ may be obtained as $\{x\in A\mid f(x)=0\}$ with the induced operations, and $k\colon {K\to A}$ as the canonical inclusion. Since $f$ is a morphism of algebras, for $x$, $y\in K$ and $\lambda \in \K$ we have $f(x+y)=f(x)+f(y)=0$, $f(\lambda x)=\lambda f(x)=0$ and $f(x\cdot y)=f(x)\cdot f(y)=0$, so that $K$ is a $\K$-algebra. If now $h\colon C\to A$ is a morphism such that $f\circ h=0$, then we must define $h'\colon C\to K$ as the map which sends $x\in C$ to $h(x)\in K$---note that $f(h(x))=0$. This is the only function which makes the triangle commute, and it is a morphism, because $h$ is a morphism.

\begin{proposition}
A kernel of a morphism of non-associative algebras is computed as the kernel of the underlying linear map.
\end{proposition}

Reversing the arrows, we find the definition of a cokernel:

\begin{definition}\label{Def:Cokernel}
In a pointed category $\C$, an arrow $q\colon {B\to Q}$ is a \defn{cokernel} of an arrow $f\colon {A\to B}$ when $q\circ f=0$, and every other arrow $h\colon {B\to C}$ such that $h\circ f=0$ factors uniquely through $q$ via a morphism $h'\colon {Q\to C}$ such that $h'\circ q=h$.
\[
\xymatrix@R=1ex{&& Q \\
A \ar[r]^-f & B \ar[ru]^-q \ar[rd]_-{\forall h} \\
&& C \ar@{<--}[uu]_-{\exists !h'} }
\] 
\end{definition}

In a variety of non-associative algebras, any arrow has a cokernel, but its construction is slightly more complicated and needs some preliminary work.

\section{Kernels and ideals, cokernels and quotients}

In varieties of non-associative algebras, arrows which are kernels can be characterised as \emph{ideals}. An ideal of an algebra is like an ideal of a ring, or a normal subgroup of a group, both of which admit a similar categorical characterisation. This is the first place where we see that the added multiplication which distinguishes an algebra from a mere vector space makes an actual difference: \emph{not every subalgebra can occur as a kernel}.

\begin{definition}\label{Def:Ideal}
Given an algebra $A$, a \defn{subalgebra} of $A$ is a subspace $S$ of~$A$ which is closed under multiplication, so $SS\subseteq S$. 

An \defn{ideal} $I$ of $A$ is a subalgebra such that $AI\subseteq I\supseteq IA$: for all $a\in A$ and $x\in I$, the products $ax$ and $xa$ in $A$ are still elements of $I$. 
\end{definition}

For a given algebra in $\V$, all of its subalgebras are $\V$-algebras as well, since the multiplication on $A$ restricts to a multiplication on $S$, which of course makes the same equations hold. In fact, a subset $S\subseteq A$ is a subalgebra precisely when the canonical injection $s\colon {S\to A}$ is a morphism in $\V$.

\begin{example}\label{Ex:AssocPoly}
Write $\K\langle x\rangle$ for the set of associative polynomials with zero constant term in $x$, which is the $\K$-vector space with basis $\{x,x^2,\dots,x^n,\dots\}$, equipped with the obvious associative multiplication determined by $x^mx^n=x^{m+n}$. Then the vector space generated by the even degrees $\{x^{2k}\mid k\geq 1\}$ of $x$ forms a subalgebra of $\K\langle x\rangle$ which is not an ideal. 
\end{example}

There is a one-to-one correspondence between the set of all ideals of $A$ and the set of all equivalence relations $\sim$ on $A$ compatible with the operations of $A$, that is, such that, for every $a$, $b$, $c$, $d\in A$ and every $\lambda\in\K$, $a\sim b$ and $c\sim d$ implies $a+c\sim b+d$, $ac\sim bd$ and $\lambda a\sim \lambda b$. The following result explains this in detail.

\begin{proposition}\label{Prop:Quotient-of-Ideal}
For a subalgebra $K$ of an algebra $A$, let $k\colon {K\to A}$ denote the canonical inclusion. Let $q\colon {A\to A/K}\colon a\mapsto a+K$ denote the canonical linear map to the quotient vector space $A/K=\{a+K\mid a\in A\}$.

The following conditions are equivalent:
\begin{enumerate}
\item $K$ is an ideal;
\item there is a unique algebra structure on $A/K$ for which $q\colon {A\to A/K}$ becomes a morphism of algebras; then $k$ is the kernel of $q$, and $q$ is the cokernel of~$k$;
\item $k$ is the kernel of some morphism of algebras $f\colon {A\to B}$;
\item $K$ is the equivalence class $[0]_\sim$ of the zero of $A$ modulo some equivalence relation $\sim$ compatible with the operations of the algebra $A$.
\end{enumerate}
A map $k\colon {K\to A}$ satisfying these conditions is called a \defn{normal monomorphism}.
\end{proposition}
\begin{proof}
3.\ implies 1.: if $k$ is a kernel of $f$ then for all $x\in K$ and $a\in A$ we have $f(xa)=f(x)f(a)=0f(a)=0$. Likewise, $f(ax)=0$, so that $xa$ and $ax$ are in $K$. Hence $K$ is an ideal of $A$. 

4.\ implies 1.\ because if $a\sim0$, $b\sim0$ and $\lambda\in\K$ then $a+b\sim0+0=0$, $ab\sim00=0$ and $\lambda a\sim\lambda0=0$.

Clearly, 2.\ implies 3.; let us prove that 1.\ implies 2.. For $q$ to become a morphism of algebras, we have no choice but to put $(a+K)\cdot (b+K)=ab+K$ for $a$, $b\in A$. Clearly then, $k$ will be the kernel of $q$. 

We need to prove that this multiplication on $A/K$ is indeed a $\K$-algebra structure. First of all, it is well defined: if $a+K=a'+K$, then
\[
(a+K)\cdot (b+K)=ab+K=ab+(a'-a)b+K=a'b+K=(a'+K)\cdot (b+K)\text{,}
\]
since $(a'-a)b\in KA$ is in the ideal $K$. Similarly, the multiplication does not depend on the chosen representative in the second variable. Its bilinearity follows immediately from the bilinearity of the multiplication on $A$. Finally, any equation which holds in $A$ also holds in $A/K$, so the $\K$-algebra $A/K$ is an object of $\V$.

Let us now check that the morphism $q\colon {A\to A/K}$ is indeed the cokernel of $k$. Suppose $h\colon {A\to C}$ is a morphism such that $h\circ k=0$. Then we have no choice but to impose that $h'\colon {A/K\to C}$ takes a class $a+K$ and sends it to $h(a)$. The question is, whether this defines a morphism of algebras. First of all, the choice of representatives plays no role, since $a+K=a'+K$ implies $a-a'\in K$, so that $h(a)-h(a')=h(a-a')=0$. The other properties follow easily.

Finally, 3.\ implies 4.: we let $a\sim b$ when $f(a)=f(b)$; then it is easily seen that $\sim$ is compatible with the operations of $A$. Furthermore, $a\sim 0$ if and only if $a$ is in the kernel $K$ of $f$.
\end{proof}

In other words, \emph{ideals have quotients}, \emph{kernels have cokernels}. What about the cokernel of an arbitrary morphism of algebras? For this we need a description of the \defn{ideal of $A$ generated by a subset $S\subseteq A$}, which is the smallest ideal of the algebra $A$ that contains $S$.

\begin{proposition}\label{Prop:Generated-Ideal}
Given any subset $S$ of an algebra $A$, the ideal $I$ of $A$ generated by $S$ exists, and may be obtained as follows:
\begin{enumerate}
	\item $I$ is the intersection of all ideals of $A$ that contain $S$;
	\item $I$ is the union of the sequence of subspaces $I_n$ of $A$, obtained inductively as follows:
\begin{itemize}
	\item $I_0$ is the subspace of $A$, generated by $S$;
	\item $I_{n+1}$ is the subspace of $A$, generated by $I_n$, $AI_n$ and $I_nA$.
\end{itemize} 
\end{enumerate}
In other words, each element of $I$ can be obtained as a linear combination of elements of $S$ and products of elements $S$ with elements of~$A$.
\end{proposition}
\begin{proof}
For the proof of 1., consider the set of ideals 
\[
\J=\{J\subseteq A\mid\text{$J$ is an ideal of $A$ containing $S$}\}\text{.}
\]
This set is non-empty because it contains $A$. Any intersection of a set of subspaces of a vector space is still a subspace. If those subspaces happen to be ideals, then the result is still an ideal, since for $x\in \bigcap \J$ and $a\in A$, the products $ax$ and $xa$ are elements of all members of the set $\J$, so they are in its intersection. It follows that $\bigcap \J$ is itself an element of $\J$, and clearly it is the smallest element.

For 2.\ we have to prove that the subset $L=\bigcup_{n\in \N} I_n$ of~$A$ is an element of $\J$. Note that the elements of $I_{n+1}$ are linear combinations of products of elements of the form $ax$, $xa$ and $x$, where $a\in A$ and $x\in I_n$. It is easily seen that this set $L$ is still a subspace of $A$, and it is also obvious that $L$ is an ideal. Note that \emph{any} ideal of $A$ that contains $S$ necessarily also contains the other elements of $L$, so $I$ and $L$ must coincide. 
\end{proof}

Note that the ideal generated by a subset $S$ of an algebra $A$ may be obtained as the smallest ideal that contains the subspace of $A$ spanned by~$S$.

\begin{proposition}\label{Prop:Cokernel}
In any variety of non-associative algebras, given a morphism $f\colon {A\to B}$, a cokernel $q\colon {B\to Q}$ of it exists, and may be obtained as follows:
\begin{enumerate}
	\item take the image $f(A)=\{f(a)\mid a\in A\}$, this is a subalgebra of $B$;
	\item take the smallest ideal $I$ of $B$ containing $f(A)$;
	\item let $q$ be the quotient $B\to B/I$.
\end{enumerate}
\end{proposition}
\begin{proof}
Propositions~\ref{Prop:Quotient-of-Ideal} and~\ref{Prop:Generated-Ideal} already tell us that this procedure can indeed be carried out. Actually, 1.\ needs to be checked separately, but this is easy. We only have to show that the result $q$ is a cokernel of $f$.

Since we know that $q$ is a cokernel of the canonical inclusion $i\colon{I\to B}$, we only need to prove that carries out the same role for $f$. This amounts to showing that for any morphism $h\colon{B\to C}$, we have $h\circ f=0$ if and only if $h\circ i=0$. Note that since $f(A)\subseteq I$, there exists $f'\colon A\to I\colon a\mapsto f(a)$ such that $i\circ f'=f$. Hence $h\circ i=0$ implies $h\circ f=h\circ i\circ f'=0$. For the converse, we know that $h$ vanishes on all elements of $I$ of the form $f(a)$ for $a\in A$, and need to extend this to all of~$I$. This is clear however, since each element of $I$ is linear combination of elements of~$f(A)$ and products of elements $f(A)$ with elements of~$B$.
\end{proof}

\section{Short exact sequences and protomodularity}
One of the key notions of Homological Algebra is the concept of a \emph{(short) exact sequence}. It can be defined in any pointed category which has kernels and cokernels, so in particular in all varieties of non-associative algebras. As it turns out, here the concept is particularly well behaved, since the \emph{Split Short Five Lemma} holds. This implies that a variety of non-associative algebras is always a \emph{protomodular category}---a central notion in Categorical Algebra, and part of the definition of a \emph{semi-abelian} category.

\begin{definition}\label{Def:SES}
In a pointed category, a \defn{short exact sequence} is a couple of composable morphisms $(f\colon {A\to B},g\colon {B\to C})$ where $f$ is a kernel of $g$ and $g$ is a cokernel of $f$. 
\end{definition}

This situation is usually pictured as a sequence
\begin{equation}\label{SES}
	\xymatrix{0 \ar[r] & A \ar[r]^-f & B \ar[r]^g & C \ar[r] & 0\text{.}}\tag{$\star$}
\end{equation}
Knowing that a couple of composable morphisms is a short exact sequence encodes certain information about the objects involved. This is precisely the type of information that is dealt with by Homological Algebra. 

In a category of vector spaces, for instance, the exactness of this sequence not only says that $C\cong B/A$; it also implies that $B$ is isomorphic to the direct sum~${A\oplus C}$: up to isomorphism, the outer objects completely determine the middle one. This is a consequence of the next result (Theorem~\ref{Thm:Proto}), which is valid in any variety of non-associative algebras. We first need a definition.

\begin{definition}
A morphism $f\colon A\to B$ in a category $\C$ is said to be a \defn{split epimorphism} when there exists a morphism $s\colon {B\to A}$ such that $f\circ s=1_B$. 

Dually, a morphism $s\colon B\to A$ in $\C$ is called a \defn{split monomorphism} when there exists a morphism $f\colon {A\to B}$ such that $f\circ s=1_B$. 
\end{definition}

As we see, split epimorphisms and split monomorphisms occur together: the splitting $s$ of a split epimorphism $f$ is a split monomorphism, and vice versa. 

\begin{example}
In $\Set$, any injection $s\colon {B\to A}$ where $B\neq \emptyset$ is a split monomorphism. The statement that ``any surjection $f\colon {A\to B}$ is a split epimorphism'' is equivalent to the Axiom of Choice, which allows us to define $s\colon {B\to A}$ by picking, for each $b\in B$, an element $a$ in $f^{-1}(b)$, and calling this $s(b)$.
\end{example}

\begin{example}
Under the the Axiom of Choice, every vector space has a basis. Then in $\Vect_\K$, every surjective linear map $g\colon {B\to C}$ is a split epimorphism. A~splitting $s\colon C\to B$ for $g$ may be defined as follows. Let $Y$ be a basis of $C$. Then the restriction of $g$ to a function $g^{-1}(Y)\to Y$ is a surjection. Hence it is a split epimorphism. Let $t\colon Y\to B$ denote a splitting, viewed as a function with codomain~$B$. Since~$Y$ is a basis of $C$, the function $t$ extends to a linear map $s\colon {C\to B}$.
\end{example}

Outside these two examples, split monomorphisms and split epimorphisms tend to be quite scarce. For a given morphism of non-associative algebras, to belong to one of those classes is a strong condition with serious consequences. 

\begin{theorem}[Split Short Five Lemma]\label{Thm:Proto}
In a variety of non-associative algebras, consider the diagram
\[
\xymatrix{A \ar[d]_-\alpha \ar[r]^-f & B \ar[d]^-\beta \ar@<.5ex>[r]^-g & C \ar@<.5ex>[l]^-s \ar[d]^-\gamma\\
D \ar[r]_-k & E \ar@<.5ex>[r]^-q & F \ar@<.5ex>[l]^-t}
\]
where $f$ is a kernel of $g$ and $k$ is a kernel of $q$, where $g\circ s=1_C$ and $q\circ t=1_F$, and where the three squares commute: $\beta\circ f=k\circ \alpha$, $q\circ \beta=\gamma\circ g$ and $\beta\circ s=t\circ \gamma$. If $\alpha$ and~$\gamma$ are isomorphisms, then $\beta$ is an isomorphism as well.
\end{theorem}
\begin{proof}
By Lemma~\ref{Lem:Iso}, it suffices that $\beta$ is injective and surjective. Consider $e\in E$, then $e-tq(e)$ is sent to $0$ by $q$, so there is a $d\in D$ such that $e=k(d)+tq(e)$. Take $a=\alpha^{-1}(d)$, $c=\gamma^{-1}(q(e))$ and $b=f(a)+s(c)$. Then 
\[
\beta(b)=\beta(f(a)+s(c))=k(\alpha(\alpha^{-1}(d)))+t(\gamma(\gamma^{-1}(q(e))))=e\text{,}
\]
which proves that $\beta$ is surjective. 

We now use the fact that $\beta^{-1}(0)=\{0\}$ if and only if $\beta$ is injective. To prove the injectivity of $\beta$, we may let $b\in B$ be such that $\beta(b)=0$. Since then also $0=q(\beta(b))=\gamma(g(b))$, and since $\gamma$ is an injection, $g(b)=0$. Hence there is an $a\in A$ such that $f(a)=b$. Now $k(\alpha(a))=\beta(f(a))=\beta(b)=0$, while both $k$ and $\alpha$ are injections, so $a=0$. It follows that $b=f(a)=0$. 
\end{proof}

Note that in this proof we are only using the (additive) group structure of the algebras. This is not a coincidence, as the result is valid in contexts which are much more general than the one where we are working now.

\begin{corollary}
	In a short exact sequence of vector spaces such as \eqref{SES}, the object~$B$ is isomorphic to $A\oplus C$.
\end{corollary}
\begin{proof}
Being a cokernel, the morphism $g$ is a surjection, which implies that it is a split epimorphism of vector spaces. Let $s\colon C\to B$ be a splitting for $g$, and consider the diagram
\[
\xymatrix{A \ar@{=}[d] \ar[r]^-{\langle1_A,0\rangle} & A\oplus C \ar[d]^-\beta \ar@<.5ex>[r]^-{\pi_C} & C \ar@<.5ex>[l]^-{\langle0,1_C\rangle} \ar@{=}[d]\\
A \ar[r]_-f & B \ar@<.5ex>[r]^-g & C \ar@<.5ex>[l]^-s}
\]
where $\beta(a,c)=f(a)+s(c)$ for $a\in A$, $c\in C$. The result now follows from Theorem~\ref{Thm:Proto}.
\end{proof}

Nothing like this is true for algebras, though. In general it is much harder to recover the middle object in a short exact sequence. 

\begin{exercise}
	Find an example of a short exact sequence in which the middle object does not decompose as a direct sum of the outer objects.
\end{exercise}

On the other hand, Theorem~\ref{Thm:Proto} says that \emph{any variety of non-associative algebras is \defn{Bourn protomodular}}, and as such it has important consequences, some of which we shall encounter later on. 

In order for us to give the definition of a (potentially long) \emph{exact sequence}, we need the category to have a richer structure. In this stronger context, we will also be able to extend Theorem~\ref{Thm:Proto} to general (non-split) short exact sequences.

To proceed, we need a more precise view on precisely which kind of an equation may determine a variety of non-associative algebras. It turns out that for this, the notion of a \emph{non-associative polynomial} is crucial. We use it to determine some important adjunctions which occur in this context.

\section{Polynomials and free non-associative algebras}

\begin{definition}
A \defn{magma} is a set $X$ equipped with a binary operation
\[
\cdot\colon X\times X\to X\colon (x,y)\mapsto x\cdot y=xy\text{.}
\]
A morphism of magmas $f\colon(X,\cdot )\to (Y,\cdot)$ is a function $f\colon X\to Y$ which preserves the multiplication: $f(x\cdot x')=f(x)\cdot f(x')$ for all $x$, $x'\in X$. Magmas and their morphisms form a category denoted $\Mag$.
\end{definition}

Like for non-associative algebras, the multiplication in magma $(X,\cdot)$ need not satisfy any additional rules such as associativity or the existence of a unit. On the other hand, any (abelian) group or (commutative) monoid has an underlying magma structure, and a non-associative algebra has two such structures (associated with $+$ and $\cdot$).

\begin{definition}
Let $S$ be a set. A \defn{non-associative word} in the alphabet $S$ is a finite sequence of elements of $S$ (called the \defn{letters} of the word) and brackets ``$($'' and ``$)$'' of the following kinds (and no others):
\begin{enumerate}
\item for every element $s$ of $S$, $(s)$ is a non-associative word;
\item if $w_1$ and $w_2$ are non-associative words, then the string $(w_1w_2)$ is a non-associative word.
\end{enumerate}
To avoid a rapidly increasing number of brackets, we will not write the outer brackets in a non-associative word.

We write $M(S)$ for the set of non-associative words in the alphabet $S$. The \defn{length} of a word is the number of letters (in $S$, brackets don't count) that it consists of, and the \defn{degree} of a letter the number of times it occurs in the given word. We sometimes write $\varphi(x_{1},\dots,x_{n})$ for a word in which the elements $x_{1}$, \dots, $x_{n}$ of $S$ (and no others) appear.

The rule 2.\ allows us to define a binary operation $\cdot$ on $M(S)$ making it into a magma which we call the \defn{free magma} on $S$. The canonical inclusion of $S$ into $M(S)$ obtained from 1.\ is written $\eta_S\colon {S\to M(S)}$.
\end{definition}

\begin{example}
When $S=\{x,y,z\}$, the strings $x$, $xy$, $(xy)z$, $x(yz)$, $x(xy)$, $(x(yz))x$ and $(xy)(zx)$ are elements of $M(S)$. The strings $x(yz)x$ and $xyz$ are not in $M(S)$, and neither is the string $()$. (This ``empty string'' appears when considering free \emph{unitary} magmas.)
\end{example}

\begin{proposition}
	For every magma $(X,\cdot)$ and every function $f\colon {S\to X}$ there exists a unique morphism of magmas $\overline{f}\colon (M(S),\cdot)\to (X,\cdot)$ such that the triangle of functions
\[
\xymatrix@!0@=4em{S \ar[rr]^-{\eta_S} \ar[rd]_-{f} && M(S) \ar@{-->}[ld]^-{\overline{f}}\\
&X}
\]
commutes.
\end{proposition}
\begin{proof}
The function $\overline{f}$ must send a word $(s)$ where $s\in S$ to $f(s)$ in order to make the triangle commute. It must preserve products, so a string $w_1w_2$ is sent to $f(w_1)\cdot f(w_2)$.
\end{proof}

The free magma construction conspires to a functor $M\colon\Set\to \Mag$ which sends a set $S$ to $M(S)$, and a function $f\colon {S\to T}$ to the morphism of magmas $M(f)\colon M(S)\to M(T)$ induced by $\eta_T\circ f\colon {S\to M(T)}$. On the other hand, there is the \defn{forgetful functor} $\Mag\to \Set$ which forgets about multiplications, taking a magma and sending it to its underlying set. 

This situation fits the following general definition, of one of the key concepts in Category Theory:

\begin{definition}\label{Def:Adj}
Consider a pair of functors $L\colon{\C\to \D}$ and $R\colon{\D\to \C}$. Then $L$ is said to be \defn{left adjoint} to $R$, and $R$ is said to be \defn{right adjoint} to $L$, when for every object $C$ of $\C$ there is a morphism $\eta_C\colon {C\to R(L(C))}$ in $\C$, such that for every object $D$ of $\D$ and every morphism $f\colon {C\to R(D)}$ there exists a unique morphism $\overline{f}\colon {L(C)\to D}$ in $\D$ such that the triangle
\[
\xymatrix@!0@=4em{C \ar[rr]^-{\eta_C} \ar[rd]_-{f} && R(L(C)) \ar@{-->}[ld]^-{R(\overline{f})}\\
&R(D)}
\]
commutes in $\C$. This is often denoted in symbols as $L\dashv R$. When they exist, adjoints are unique (up to isomorphism), so we may say that $R$ \defn{has a left adjoint} or $L$ \defn{has a right adjoint}. The collection of morphisms $(\eta_C)_C$ is called the \defn{unit} of the adjunction, and always forms a \defn{natural transformation} from $1_\C$ to $R\circ L$, which means that for every morphism $c\colon{C\to C'}$ in $\C$, the square
\[
\xymatrix{C \ar[d]_-{c} \ar[r]^-{\eta_C} & R(L(C)) \ar[d]^-{R(L(c))}\\
C' \ar[r]_-{\eta_{C'}} & R(L(C'))}
\]
in $\C$ is commutative.
\end{definition}

In other words, the free magma functor is left adjoint to the forgetful functor to~$\Set$, and in fact this is the reason why it carries that name: it plays the same role as the free group functor, for instance.

There are many equivalent ways to phrase adjointness, and going into the general theory of adjunctions here would lead us too far. However, we will meet several further examples, starting at once with the following one, which makes one of the relationships between non-associative algebras and magmas explicit:

\begin{example}\label{Ex:MagAlgFunctor}
For any field $\K$, the forgetful functor $\Alg_\K\to \Mag$ which takes an algebra $(A,\cdot)$ and sends it to the underlying set of the vector space $A$, equipped with the multiplication $\cdot$, has a left adjoint denoted $\K[-]\colon{\Mag\to \Alg_\K}$ and called the \defn{magma algebra functor}. (It is a variation on the \emph{group algebra functor} which plays a similar role for groups and cocommutative Hopf algebras.) 

The functor $\K[-]$ takes a magma $(X,\cdot)$ and sends it to the $\K$-vector space $\K[X]$ with basis $X$, whose elements are finite linear combinations of the elements of $X$, equipped with the multiplication $\cdot \colon \K[X]\times \K[X]\to \K[X]$ defined by 
\[
\bigl(\sum_{i=1}^n\lambda_ix_i,\sum_{j=1}^k\mu_jy_j\bigr)\mapsto \sum_{i=1}^n\sum_{j=1}^k\lambda_i\mu_j (x_i\cdot y_j)
\]
for $x_i$, $y_j\in X$ and $\lambda_i$, $\mu_j\in \K$. Note that its bilinearity is obvious.

$\K[-]$ satisfies the universal property of a left adjoint, because for the natural inclusion $\tilde\eta_{(X,\cdot)}\colon {(X,\cdot)\to (\K[X],\cdot)}$, we have that any given morphism of magmas $f\colon X\to A$, where $(A,\cdot)$ is a non-associative algebra, extends to a unique morphism of algebras $f'\colon \K[X]\to A$ such that $f'\circ \tilde\eta_X = f$ in $\Mag$. Indeed, this just follows from the fact that $X$ is a basis of~$\K[X]$ and the definition of the multiplication of that algebra. As in the case of magmas, this determines how the functor $\K[-]$ should act on morphisms.
\end{example}

\begin{example}\label{Ex:AdjComp}
	\emph{Adjunctions compose}, and thus we find the construction of the \defn{free non-associative $\K$-algebra} on a set. 
\[
\xymatrix{\Set \ar@<1ex>[r]^-{M} \ar@{}[r]|-{\perp} & \Mag \ar@<1ex>[r]^-{\K[-]} \ar@{}[r]|-{\perp} \ar@<1ex>[l]^-{\Forget} & \Alg_\K \ar@<1ex>[l]^-{\Forget}}
\]
The functors to the left first forget the vector space structure of an algebra $A$, then the multiplication of the underlying magma $(A,\cdot)$, so that we obtain the underlying set of $A$. Looking at the diagram in $\Set$
\[
\xymatrix@!0@C=6em@R=4em{S \ar[r]^-{\eta_S} \ar[rd]_-{f} & M(S) \ar@{-->}[d]^-{\overline{f}} \ar[r]^-{\tilde\eta_{M(S)}} & \K[M(S)] \ar@{-->}[ld]^-{\overline{f}'} \\
&A}
\]
it is easy to see that the composite functor to the right does indeed satisfy the universal property of a left adjoint. 
\end{example}

\begin{exercise}\label{Exer:AdjComp}
	Use this idea to prove that \emph{adjunctions compose} in general: given functors $L\colon{\C\to \D}$, $L'\colon{\D\to \E}$ and $R\colon{\D\to \C}$, $R'\colon{\E\to \D}$ such that $L\dashv R$ and $L'\dashv R'$, show that $L'\circ L\dashv R\circ R'$.
\end{exercise}

For a given set $S$, an element of $\K[M(S)]$ is a $\K$-linear combination of non-associative words in the alphabet $S$. In other words:

\begin{definition}\label{Def:Polynomial}
For a given set $S$, a \defn{(non-associative) polynomial} with variables in~$S$ is an element of the free $\K$-algebra on $S$. For the sake of simplicity, we write $\Free{S}$ for the algebra $\K[M(S)]$.
	
A \defn{monomial} in $\Free{S}$ is any scalar multiple of an element of~$M(S)$. The \defn{type} of a monomial $\varphi(x_{1},\dots,x_{n})$ is the element $(k_{1},\dots,k_{n})\in \N^{n}$ where $k_{i}$ is the degree of $x_{i}$ in $\varphi(x_{1},\dots,x_{n})$. A polynomial is \defn{homogeneous} if its monomials are all of the same type. Any polynomial may thus be written as a sum of homogeneous polynomials, which are called its \defn{homogeneous components}.
\end{definition}

\begin{remark}
Our algebras need not have units, and so our polynomials have no constant terms. 
\end{remark}

\begin{example}
When $S=\{x,y,z\}$ and $\K$ is any field, $x$, $xy$, $xx$, $(xx)x$, -$x(xx)$, $(x(yz))x$, $xy+yx$, $x(yz)-(xy)z$, $x(yz)+z(xy)+ y(zx)$ and $xy-yx+(xy)z$ are polynomials over $S$. The first six of those are monomials, the next three are homogeneous polynomials (of respective types $(1,1)$, $(1,1,1)$ and $(1,1,1)$ in $(x,y,z)$), while the last one is not homogeneous, and its homogeneous components are $xy-yx$ and $(xy)z$. 
\end{example}

\section{Varieties of non-associative algebras}\label{S8}
The next step is to understand what is a variety of non-associative algebras from the categorical viewpoint. We first make Definition~\ref{Def:Variety} fully precise. What was missing there is an explicit description of what exactly is an equational condition that determines a variety. We now see that it is given by an algebra of non-associative polynomials, an ideal of a free algebra.

We fix a countable set of variables $X=\{x_1,x_2,\dots,x_n,\dots\}$ and consider the free $\K$-algebra $\Free{X}$. An element $\psi$ of this algebra is a polynomial in a finite number of variables, say $x_1$, \dots, $x_n$. (In concrete examples we often prefer using letters $x$, $y$, $z$, etc.\ for the variables in a polynomial.) 

Given elements $a_1$, \dots, $a_n$ in an algebra~$A$, the universal property of free algebras gives us a unique algebra morphism $z\colon\Free{X}\to A$ which sends $x_i$ to $a_i$ if $1\leq i\leq n$, and to $0$ otherwise. We shall write $\psi(a_1,\dots,a_n)$ for $z(\psi)$, and say that the element $\psi(a_1,\dots,a_n)$ of~$A$ is obtained by \defn{substitution} of the variables $x_1$, \dots, $x_n$ in the non-associative polynomial $\psi(x_1,\dots,x_n)$ by elements $a_1$, \dots, $a_n$ of $A$. In some sense, this process \emph{evaluates} the polynomial in $a_1$, \dots, $a_n$.

\begin{lemma}[Algebra morphisms preserve polynomials]\label{Lem:Subst}
	Let $f\colon{A\to B}$ be a $\K$-algebra morphism and $\psi$ a polynomial in $n$ variables. Then for all $a_1$, \dots, $a_n\in A$, we have $f(\psi(a_1,\dots,a_n))=\psi(f(a_1),\dots,f(a_n))$. 
\end{lemma}
\begin{proof}
	This follows immediately from the definition of substitution in a polynomial. If $z\colon\Free{X}\to A$ determines substitution by $a_1$, \dots, $a_n$ in $A$, then $f\circ z$ determines substitution by $f(a_1)$, \dots, $f(a_n)$ in $B$; now $f(z(\psi))=f(\psi(a_1,\dots,a_n))$ is $\psi(f(a_1),\dots,f(a_n))$ by definition.
\end{proof}

This allows us to extend Proposition~\ref{Prop:Generated-Ideal} and give yet another description of the ideal of an algebra, generated by a subset.

\begin{lemma}\label{Lem:Generated-Ideal}
Given any subset $S$ of an algebra $A$, the ideal generated by $S$ is the vector space spanned by all $\varphi(a_1,\dots, a_n)$, where $\varphi$ is a mononomial (a scalar multiple of a word) and $a_1$, \dots, $a_n$ are elements of $A$ with at least one of the $a_i$ in~$S$.
\end{lemma}
\begin{proof}
The vector space spanned by these elements is certainly an ideal of $A$, because if $\varphi(x_1,\dots, x_n)$ is a monomial, then so are the products $x_{n+1}\varphi(x_1,\dots, x_n)$ and $\varphi(x_1,\dots, x_n)x_{n+1}$.

Now suppose $I$ is an ideal of $A$ that contains $S$. By Lemma~\ref{Lem:Subst}, the quotient map $q\colon A\to A/I$ sends $\varphi(a_1,\dots, a_n)$ to $\varphi (q(a_1),\dots,q(a_n))$. This element of $A/I$ is zero, because $\varphi$ is a monomial and one of the $a_i$ is in $S\subseteq I$. Hence $\varphi(a_1,\dots, a_n)$ is in the kernel $I$ of $q$. It follows from item 1.\ in Proposition~\ref{Prop:Generated-Ideal} that $\varphi(a_1,\dots, a_n)$ is a member of the ideal of $A$, generated by $S$.
\end{proof}

\begin{definition}
	A non-associative polynomial $\psi=\psi(x_1,\dots,x_n)$ is called an \defn{identity} of an algebra $A$ if $\psi(a_1,\dots,a_n)=0$ for all $a_1$, \dots, $a_n\in A$. We also say that~$A$ \defn{satisfies the identity} $\psi$ or that the identity $\psi$ \defn{is valid} in~$A$. 
	
Let $I$ be a subset of $\Free{X}$. The class of all algebras that satisfy all identities in~$I$ is called the \defn{variety of $\K$-algebras determined by $I$}.
\end{definition}

 As in Definition~\ref{Def:Variety}, we thus obtain a full subcategory of $\Alg_\K$. The main difference is that here we focus on polynomials instead of equations; for instance, instead of expressing commutativity as an equation $xy=yx$ that all pairs of elements $x$,~$y$ of an algebra $A$ must satisfy, we write the condition as $xy-yx=0$, and notice that the expression $xy-yx$ is a polynomial in $x$ and $y$, actually an element of a free non-associative algebra.

\begin{definition}\label{Def:T-ideal}
A \defn{$T$-ideal (over $\K$)} is an ideal of $\Free{X}$ which is closed under substitution. 
\end{definition}

The set of all polynomial identities of an arbitrary variety of non-associative algebras forms a $T$-ideal. Conversely, given a $T$-ideal, we may consider the variety determined by the identities in the ideal. This gives rise to a one-to-one correspondence between the $T$-ideals over $\K$ and the varieties of non-associative algebras over~$\K$. 

\begin{definition}
Let $I$ be a subset of $\Free{X}$. For any non-associative $\K$-algebra $A$, we let $I(A)$ be the ideal of $A$ generated by all elements of the form $\psi (a_1,\dots,a_n)$, where $\psi\in I$ and $a_1$, \dots, $a_n\in A$. 
\end{definition}

By definition, for any subset $I$ of $\Free{X}$, the ideal $I(\Free{X})$ is a $T$-ideal. In fact, $I(\Free{X})$ is the set of all identities that hold in the variety of $\K$-algebras determined by a given set of polynomials~$I$.

\begin{proposition}\label{Prop:Reflective}
	Let $\V$ be a variety of $\K$-algebras determined by a set of polynomials~$I$. Then the inclusion functor $\V\to \Alg_\K$ has a left adjoint $L\colon {\Alg_\K\to \V}$.
\end{proposition}
\begin{proof}
	The functor $L$ sends an object $A$ to the quotient $L(A)=A/I(A)$, so that for each $A$ we have a short exact sequence
\[
\xymatrix{0 \ar[r] & I(A) \ar[r]^-{\mu_A} & A \ar[r]^-{\eta_A} & L(A) \ar[r] & 0}
\] 
in $\Alg_\K$. The algebra $A/I(A)$ is indeed an object of $\V$: if $\psi(x_1,\dots,x_n)\in I$ and $a_1$, \dots, $a_n\in A$, then
\[
\psi(a_1+I(A),\dots,a_n+I(A))=\psi(a_1,\dots,a_n)+I(A)=I(A)
\]
since $\psi(a_1,\dots,a_n)\in I(A)$, and by the definition of the operations in $A/I(A)$. So $L(A)$ satisfies all identities of $I$.

Suppose $B$ is an algebra in $\V$, and let $f\colon{A\to B}$ be a morphism in $\Alg_\K$. Then $f\circ \mu_A=0$, because for every element $\psi(a_1,\dots,a_n)$ of $I(A)$,
\[
f(\psi(a_1,\dots,a_n))=\psi(f(a_1),\dots,f(a_n))\in B 
\]
by Lemma~\ref{Lem:Subst}, which is zero because $B$ satisfies all identities in $I$. Hence the morphism $f$ factors uniquely through the cokernel $\eta_A$ of $\mu_A$, which proves the universal property of the left adjoint $L$. (Again, like in Example~\ref{Ex:MagAlgFunctor}, the action of~$L$ on morphisms making it a functor is determined by the universal property.)
\end{proof}

\begin{exercise}
	Prove that $A\mapsto I(A)$ determines a functor $I\colon {\Alg_\K\to \Alg_\K}$. 
\end{exercise}

In other words, $\V$ is a \emph{reflective subcategory} of $\Alg_\K$:

\begin{definition}
	A full subcategory $\D$ of a category $\C$ is called a \defn{reflective subcategory} when the inclusion functor ${\D\to \C}$ has a left adjoint.
\end{definition}

\begin{exercise}
When its inclusion functor has a right adjoint, a subcategory is said to be \defn{coreflective}. Look up examples of this situation.
\end{exercise}

Again using that adjunctions compose, we now see that free $\V$-algebras exist.

\begin{corollary}
For any set $S$, the \defn{free $\V$-algebra} on $S$ exists, and is given by $L(\Free{S})$: it is the algebra of non-associative polynomials in the alphabet $S$, modulo the identities that determine $\V$. 	
\end{corollary}
\begin{proof}
We may proceed as in Example~\ref{Ex:AdjComp} and Exercise~\ref{Exer:AdjComp}. The bottom composite functor in the diagram
\[
\xymatrix{\Set \ar@<1ex>[r]^-{\Free{-}} \ar@{}[r]|-{\perp} & \Alg_\K \ar@<1ex>[l]^-{\Forget} \ar@<1ex>[r]^-{L} \ar@{}[r]|-{\perp} & \V \ar@<1ex>[l]^-{\supseteq}}
\]
is the forgetful functor to $\Set$. Its left adjoint is the above composite functor.
\end{proof}

\begin{exercise}
We know that the variety $\AssocAlg_\K$ of associative algebras is determined by the identity $x(yz)-(xy)z$. Put $I=\{x(yz)-(xy)z\}$. Prove that the free associative algebra on a singleton set $\{x\}$, which we know is the quotient of $\Free{x}\coloneq\Free{\{x\}}$ by $I(\Free{x})$, is isomorphic to the $\K$-algebra $\K\langle x\rangle$ of \emph{associative} polynomials with zero constant term in $x$ from Example~\ref{Ex:AssocPoly}. 	
\end{exercise}

\section{Regularity, exact sequences}

We already used the image $f(A)\subseteq B$ of a morphism of algebras $f\colon{A\to B}$ and noticed that it is always a subalgebra (Proposition~\ref{Prop:Cokernel}). The existence of images can be given a general categorical treatment via the concept of a \emph{regular category}, another key ingredient in the definition of semi-abelian categories. 

\begin{definition}
	In any category, a morphism $f\colon {A\to B}$ is said to be a \defn{monomorphism} when for every pair of morphisms $a$, $b\colon{X\to A}$ such that $f\circ a=f\circ b$, we have $a=b$. 
\end{definition}

\begin{exercise}
A kernel is always a monomorphism.
\end{exercise}

\begin{proposition}
	In a variety $\V$ of non-associative algebras, a morphism is a monomorphism if and only if it is an injection.
\end{proposition}
\begin{proof}
	If $f$ is injective, then $f(a(x))=f(b(x))$ implies that $a(x)=b(x)$. So if this happens for all $x\in X$, then $a=b$. Conversely, let $c$, $d\in A$ be such that $f(c)=f(d)$. Consider the free $\V$-algebra $X=L(\Free{x})$ on a singleton set $\{x\}$, and let $a$~and~$b\colon{X\to A}$ be the algebra morphisms determined by $a(x)=c$ and $b(x)=d$. Then $f\circ a=f\circ b$, so $a=b$, hence $c=d$.
\end{proof}

The ``dual'' concept is that of an \defn{epimorphism}---a morphism $f\colon {A\to B}$ such that for every pair of morphisms $a$, $b\colon{B\to X}$ where $a\circ f=b\circ f$, we have $a=b$---but these are not well behaved in the context where we are working (see Exercise~\ref{Ex:EpiNotSurj}). Surjective algebra morphisms are captured by something which is slightly stronger, called a \emph{regular epimorphism}. To define it, we first need to generalise the concept of a cokernel (Definition~\ref{Def:Cokernel}):

\begin{definition}
In a category $\C$, an arrow $q\colon {B\to Q}$ is a \defn{coequaliser} of a pair of parallel arrows $f$, $g\colon {A\to B}$ when $q\circ f=q\circ g$, and every other arrow $h\colon {B\to C}$ such that $h\circ f=h\circ g$ factors uniquely through $q$ via a morphism $h'\colon {Q\to C}$ such that $h'\circ q=h$.
\[
\xymatrix@R=1ex{&& Q \\
A \ar@<.5ex>[r]^-f \ar@<-.5ex>[r]_-g & B \ar[ru]^-q \ar[rd]_-{\forall h} \\
&& C \ar@{<--}[uu]_-{\exists !h'} }
\] 
\end{definition}

Note that the definition of a cokernel is the special case where $g=0$. Dually, we could define \emph{equalisers} as a generalisation of kernels, but we shall not need those here.

Actually, in a variety of non-associative algebras, coequalisers always exist, and they can be obtained as cokernels:

\begin{proposition}\label{Prop:Existence-Coequalisers}
Given a pair of parallel arrows $f$, $g\colon {A\to B}$ in a variety of non-associative algebras $\V$, their coequaliser $q\colon {B\to Q}$ may be obtained as the quotient of the ideal $I$ of $B$ generated by the elements of the form $f(a)-g(a)$ for $a\in A$.
\end{proposition}
\begin{proof}
Let $q$ denote the quotient $B\to B/I$. Since $q$ sends all elements of the form $f(a)-g(a)$ to zero, we already have that $q\circ f=q\circ g$. Now consider $h\colon {B\to C}$ such that $h\circ f=h\circ g$. We obtain the needed morphism $h'$ as soon as $h$ vanishes on all of the ideal $I$. By Lemma~\ref{Lem:Generated-Ideal}, an element of $I$ is a linear combination of elements of the form $\psi(b_1,\dots,b_n)$, where $\psi$ is a monomial and one of the elements $b_1$, \dots, $b_n\in B$ is equal to $f(a)-g(a)$ for some $a\in A$. By Lemma~\ref{Lem:Subst}, 
\[
h(\psi(b_1,\dots,b_n))=\psi(h(b_1), \dots, h(b_n))=0\text{,}
\]
so $h$ factors through the quotient $q$ of $I$.
\end{proof}

\begin{definition}
In a pointed category, a \defn{normal epimorphism} is a cokernel of some morphism. A \defn{regular epimorphism} is a coequaliser of some parallel pair of morphisms.
\end{definition}

\begin{proposition}\label{Prop:Surjections}
For any morphism $h\colon {B\to C}$ in a variety of non-asso\-ciative algebras, the following conditions are equivalent:
\begin{enumerate}
	\item $h$ is a normal epimorphism;
	\item $h$ is a regular epimorphism;
	\item $h$ is a surjection.
\end{enumerate}
\end{proposition}
\begin{proof}
1.\ $\Rightarrow$ 2.\ is obvious from the definition, and 2.\ $\Rightarrow$ 1.\ is an immediate consequence of Proposition~\ref{Prop:Existence-Coequalisers}. Clearly, any quotient map $B\to B/I$ determined by an ideal $I$ is a surjection, so 1.\ implies 3. To show that 3.\ implies 1., consider a kernel $k\colon{K\to B}$ of $h$, and write $q\colon {B\to Q=B/K}$ for the cokernel of $k$. Since $h\circ k=0$, there is the unique factorisation $h'\colon {Q\to C}$ of $h$ through $q$. This $h'$ is a surjection, because $h$ is a surjection and $h'\circ q=h$. It is also an injection: let indeed $b+K\in Q$ be such that $0=h'(b+K)=h(b)$, then $b\in K$, so that $q(b)=0$. By Lemma~\ref{Lem:Iso}, this proves that $h'$ is an isomorphism.
\end{proof}

\begin{exercise}\label{Ex:EpiNotSurj}
	Find an example of an epimorphism of algebras which is not a surjection. \emph{Hint:} The canonical inclusion $\N\to \Z$ is an epimorphism of monoids.
\end{exercise}

\begin{exercise}
	Prove that, in an arbitrary category, a morphism which is both a regular epimorphism and a monomorphism is an isomorphism.
\end{exercise}

In any variety of non-associative algebras, \defn{image factorisations} exist: any morphism $f\colon {A\to B}$ may be factored into a composite $m\circ p$ of a regular epimorphism $p\colon{A\to I}$ followed by a monomorphism $m\colon{I\to B}$. This monomorphism, unique up to isomorphism, is called the \defn{image of $f$}. By the above characterisations, we can simply take $I=f(A)$, with $m$ the canonical inclusion, and $p$ the corestriction of $f$ to its image. 

The general categorical context where image factorisations are usually defined is that of a \emph{regular category}. For this we need one last (very important) concept:

\begin{definition}\label{Def:Pullback}
A commutative square 
\[
\vcenter{\xymatrix{P \ar[d]_-{\pi_A} \ar[r]^-{\pi_C} & C \ar[d]^-g \\
A \ar[r]_-f & B}}
\qquad\qquad\qquad
\vcenter{\xymatrix{X \ar@/_/[ddr]_-a \ar@/^/[drr]^-c \ar@{-->}[rd]|-{\langle a,c\rangle} \\
&P \ar[d]^-{\pi_A} \ar[r]_-{\pi_C} & C \ar[d]^-g \\
&A \ar[r]_-f & B}}
\]
in a category $\C$ is called a \defn{pullback} (of $f$ and $g$) when for every pair of morphisms $a\colon {X\to A}$, $b\colon {X\to C}$ in $\C$ such that $f\circ a=g\circ c$ there exists a unique morphism $\langle a,c\rangle\colon X\to P$ such that $\pi_A\circ \langle a,c\rangle=a$ and $\pi_C\circ \langle a,c\rangle=c$. The object $P$ is then usually written as a $B$-indexed product $A\times _BC$. The morphism $\pi_A$ is called the \defn{pullback of $g$ along $f$}, and $\pi_C$ is called the \defn{pullback of $f$ along $g$}.
\end{definition}

\begin{example}[Kernels as pullbacks]
	If $\C$ is pointed, we may consider the special case where $C=0$. Then the square is a pullback precisely when $\pi_A\colon {P\to A}$ is a kernel of $f$.
\end{example}

\begin{example}[Products]\label{Ex:Products}
	If $\C$ has a terminal object $T$, we may consider the special case where $B=T$. If then the square is a pullback, $(P,\pi_A,\pi_B)$ is called a \defn{product} of $A$ and $C$.
\end{example}

\begin{example}[Pullbacks of non-associative algebras]\label{Ex:Pb-Alg}
In any variety of non-asso\-ciative algebras, any pair of arrows $f\colon{A\to B}$, $g\colon {C\to B}$ admits a pullback. We may let $P=A\times _BC$ be the set of couples 
\[
\{(a,c)\in A\times C\mid f(a)=g(c)\}
\]
and $\pi_A$, $\pi_C$ the canonical projections. The pointwise operations $\lambda(a,c)=(\lambda a,\lambda c)$, $(a,c)+(a',c')=(a+a',c+c')$ and $(a,c)\cdot(a',c')=(aa',cc')$ make $P$ into an algebra, and for any two algebra morphisms $a$ and $c$ as above, the map which sends $x\in X$ to the couple $(a(x),c(x))\in P$ is the needed unique morphism $\langle a,c\rangle$.
\end{example}

\begin{definition}
	A category in which pullbacks, a terminal object, and coequalisers exist is said to be a \defn{regular category} when any pullback of a regular epimorphism is again a regular epimorphism. 
\end{definition}

In other words, if in the square of Definition~\ref{Def:Pullback} the morphism $g$ is a regular epimorphism, then $\pi_A$ must be a regular epimorphism as well.

\begin{proposition}
	Any variety of non-associative algebras is a regular category.
\end{proposition}
\begin{proof}
By Proposition~\ref{Prop:Surjections}, it suffices to prove that pullbacks preserve surjections. So consider a pullback as in Definition~\ref{Def:Pullback} and assume that $g$ is a surjection. We then use the description in Example~\ref{Ex:Pb-Alg} to prove that also $\pi_A$ is a surjection. For $a\in A$, consider $c\in C$ such that $g(c)=f(a)$; then the couple $(a,c)$ is in $A\times_BC$, and $\pi(a,c)=a$.
\end{proof}

\begin{exercise}
Prove that $\Set$ is a regular category.
\end{exercise}

It is a theorem of Categorical Algebra that in any regular category, image factorisations exist. We shall encounter an example of how this is used in Lemma~\ref{Lem:RegepiReflection}. For varieties of algebras we of course already knew this. On the other hand, a category which is pointed, regular and protomodular is called a \defn{homological category}, and since we now know that all varieties of non-associative algebras are such, we can apply all categorical-algebraic results known to be valid for homological categories in any variety of non-associative algebras. One example is the Short Five Lemma given here below, others are famous homological diagram lemmas such as the $3\times 3$-Lemma, a version of Noether's isomorphism theorems, etc.

\begin{theorem}[Short Five Lemma]\label{Thm:SFL}
In a variety of non-associative algebras, consider a commutative diagram with horizontal short exact sequences. 
\[
\xymatrix{0 \ar[r] & A \ar[d]_-\alpha \ar[r]^-f & B \ar[d]^-\beta \ar[r]^-g & C \ar[d]^-\gamma \ar[r] & 0 \\
0 \ar[r] & D \ar[r]_-k & E \ar[r]_-q & F \ar[r] & 0}
\]
If $\alpha$ and~$\gamma$ are isomorphisms, then $\beta$ is an isomorphism as well.
\end{theorem}

We are finally ready to extend Definition~\ref{Def:SES} to arbitrary exact sequences, which are the basic building blocks of Homological Algebra.

\begin{definition}
	In a homological category, a pair $(f\colon {A\to B},g\colon {B\to C})$ of composable morphisms is called an \defn{exact sequence} when the image of $f$ is a kernel of~$g$. A long sequence of composable morphisms is said to be \defn{exact} when any pair of consecutive morphisms is exact. 
\end{definition}

\begin{example}
	A sequence of morphisms such as \eqref{SES} on page~\pageref{SES} is exact if and only if it is a short exact sequence, which explains the notation. Exactness in $C$ means that the image of $g$ is an isomorphism, so that $g$ is a surjection. Exactness in $A$ means that the kernel of $f$ is zero, which makes $f$ an injection. Now exactness in~$B$ says that~$f$ is the kernel of $g$, so that $g$ is the cokernel of $f$.
\end{example}

\section{Semi-abelian categories}
\emph{Semi-abelian categories} were introduced by Janelidze--Márki--Tholen in 2002 in order to unify ``old'' approaches towards an axiomatisation of categories ``close to the category of groups'' such as the work of Higgins (1956) and Huq (1968) with ``new'' categorical algebra---the concepts of \emph{Barr-exactness} and \emph{Bourn-protomodularity}. Our aim is now to prove that all varieties of non-associative algebras are semi-abelian categories.

According to example Example~\ref{Ex:Products}, a \defn{product} of two objects $A$ and $C$ is a triple $(A\times C,\pi_A,\pi_C)$ that satisfies a universal property: given any pair of morphisms $a\colon X\to A$, $c\colon X\to C$, there exists a unique morphism $\langle a,c\rangle\colon {X\to A\times C}$ such that $\pi_A\circ \langle a,c\rangle=a$ and $\pi_C\circ \langle a,c\rangle=c$. Example~\ref{Ex:Pb-Alg} told us that in a variety of non-associative algebras, the product of two objects always exists, and may be obtained as the algebra of pairs $(a,c)$ where $a\in A$ and $c\in C$.

The ``dual'' concept is that of a \emph{coproduct}. It may be defined as a \emph{pushout}, which is the concept dual to that of a pullback, and has cokernels for one type of examples. The direct definition goes as follows.

\begin{definition}[Coproducts]
A \defn{coproduct} or \defn{sum} of two objects $A$ and $C$ is a triple $(A+ C,\iota_A,\iota_C)$ that satisfies the following universal property:
\[
\xymatrix{A \ar[r]^-{\iota_A} \ar[rd]_-{a} & A+C \ar@{-->}[d]|-{\left\lgroup a\;c\right\rgroup} & C \ar[l]_-{\iota_C} \ar[ld]^-{c}\\
& X}
\]
given any pair of morphisms $a\colon {A\to X}$, $c\colon {C\to X}$, there exists a unique morphism $\left\lgroup a\;c\right\rgroup\colon {A+C\to X}$ such that $\left\lgroup a\;c\right\rgroup\circ \iota_A=a$ and $\left\lgroup a\;c\right\rgroup\circ \iota_C=c$.
\end{definition}

\begin{example}
In the category $\Set$, the coproduct of two sets is their disjoint union.
\end{example}

\begin{example}
In the variety $\Alg_\K$, the coproduct of two $\K$-algebras $A$ and $C$ is obtained as follows. Let $R(B)$ denote the kernel of the morphism of algebras $\varepsilon_B\colon{\Free{B}\to B}$ which sends an element $b$ of $B$ to itself. In the free algebra~$\Free{A\dot\cup C}$ on the disjoint union $A\dot\cup C$ of $A$ and $C$, consider the ideal $J$ generated by the set $R(A)\dot\cup R(C)$. We claim that the quotient $\Free{A\dot\cup C}/J$, together with the morphisms
\[
\iota_A\colon{ A\to \Free{A\dot\cup C}/J}
\qquad\text{and}\qquad
\iota_C\colon {C\to \Free{A\dot\cup C}/J}
\]
induced by the respective inclusions of $A$ and $C$ into $A\dot\cup C$, is a coproduct of~$A$ and~$C$. Let us first show that $\iota_A$ is a morphism: for any two elements $a_1$ and $a_2$ of~$A$, the difference $a_1\cdot a_2-a_1a_2$ in $\Free{A\dot\cup C}$ of the product of $a_1$ with $a_2$ in $\Free{A\dot\cup C}$ and their product in $A$, viewed as an element of $\Free{A\dot\cup C}$, is an element of $R(A)$. The same proof works for $\iota_C$.

By Lemma~\ref{Lem:Generated-Ideal}, we now only need to prove that the arrow ${\Free{A\dot\cup C}\to X}$ sending the elements of~$A$ and~$C$ to their images through $a$ and $c$ vanishes on the elements of $R(A)$ and~$R(C)$: then it will automatically vanish on all of $J$, and thus factor through the quotient $\Free{A\dot\cup C}/J$. This, however, is immediate from the definitions of~$R(A)$ and~$R(C)$. 

A typical element of $A+C$ is thus a polynomial with variables in $A$ and $C$, for instance an expression of the form $(a_1a_2)c$. \emph{A priori} this may be interpreted in two distinct ways: either as a product $a_1\cdot a_2$ in $\Free{A\dot\cup C}$ of two elements $a_1$ and $a_2$ of~$A$, multiplied with the element $c$ of $C$; or as a product of the element $a_1a_2$ of $A$ with the element $c$ of~$C$. The quotient over $J$ ensures that these two points of view agree: $a_1\cdot a_2-a_1a_2$ is an element of $R(A)$, so $(a_1\cdot a_2-a_1a_2)c$ is in the ideal $J$.
\end{example}

\begin{proposition}
	In a variety of algebras $\V$ over a field $\K$, the coproduct of two algebras $A$ and $C$ always exists, and is obtained as the reflection into $\V$ of the sum $A+C$ in $\Alg_\K$.
\end{proposition}
\begin{proof}
	This follows immediately from the definitions of sums and reflections.
\end{proof}

We are still missing one piece of terminology which is needed for the definition of a semi-abelian category. A regular category is said to be \defn{Barr exact} when every internal equivalence relation is a kernel pair. A \defn{semi-abelian category} (in the sense of Janelidze--Márki--Tholen) is then a homological category which is Barr exact and where the coproduct of any two objects exist.

We didn't introduce internal equivalence relation, though, but we can avoid those by proving that all varieties of non-associative algebras satisfy a condition which is equivalent to Barr exactness in any homological category: \emph{the direct image of a kernel along a regular epimorphism is always a kernel}. This means that whenever we have a regular epimorphism $f\colon{A\to B}$ and a normal monomorphism $k\colon{K\to A}$, 
\[
\xymatrix{K \ar[r]^-p \ar[d]_-k & I \ar[d]^-i \\
A \ar[r]_-f & B}
\]
the image $i\colon I\to B$ of the composite $f\circ k$ is again a normal monomorphism.

\begin{theorem}
Any variety of non-associative algebras is a semi-abelian category.
\end{theorem}
\begin{proof}
We only need to prove that the direct image of an ideal along a surjective algebra morphism is not just any subalgebra, but again an ideal. Let us consider the commutative square above, where $K$ is an ideal of $A$ and $I=f(K)$. Since~$f$ is surjective, for any $b\in B$ there is an $a\in A$ such that $f(a)=b$. Now $bI=f(a)f(K)=f(aK)\subseteq f(K)=I$.
\end{proof}

\begin{exercise}
	The requirement that $f$ is a regular epimorphism is essential here. Give an example where $k$ is an ideal, but $i$ is not. \emph{Hint:} Use Example~\ref{Ex:AssocPoly}. 
\end{exercise}

As a consequence, typical constructions and results, valid in semi-abelian categories, hold in any variety of non-associative algebras. Examples are, for instance, the snake lemma, or the fact that homology of simplicial objects is captured by a Quillen model structure. Others are results in radical theory, commutator theory, or cohomology.

\section{Birkhoff subcategories}

We find a simple version of a famous theorem by Birkhoff.

\begin{definition}
A \defn{Birkhoff subcategory} $\D$ of a semi-abelian category $\C$ is a full reflective subcategory, closed under subobjects and quotients.

\defn{Closure under subobjects} means that whenever we have a monomorphism $m\colon {M\to D}$ in $\C$ where $D$ is an object of $\D$, the object $M$ is also in $\D$. \defn{Closure under quotients} means that whenever we have a regular epimorphism $q\colon {D\to Q}$ in $\C$ where $D$ is an object of $\D$, the object $Q$ is also in $\D$. 
\end{definition}

\begin{lemma}\label{Lem:RegepiReflection}
Let $\D$ be a full reflective subcategory of a regular category $\C$. Let $L\colon{\C\to \D}$ be the left adjoint of the inclusion functor. Then $\D$ is closed under subobjects in $\C$ if and only if each component $\eta_C\colon{C\to L(C)}$ of the unit $\eta$ of the adjunction is a regular epimorphism.
\end{lemma}
\begin{proof}
$\Rightarrow$ Factor $\eta_C\colon{C\to L(C)}$ as a regular epimorphism $p\colon C\to M$ followed by a monomorphism $m\colon{M\to L(C)}$ in $\C$. Then closure under subobjects tells us that $M$ is an object of $\D$, which by the universal property of $L$ gives us a unique morphism $p'\colon L(C)\to M$ such that $p=p'\circ \eta_{C}$. We see that $p'$ is an inverse to $m$, hence $m$ is an isomorphism.

$\Leftarrow$ Let $D$ be an object of $\D$ and $m\colon M\to D$ a monomorphism in $\C$. Applying the functor $L$, we find the commutative square
\[
\xymatrix{M \ar[r]^-m \ar[d]_-{\eta_M} & D \ar[d]^-{\eta_D} \\
L(M) \ar[r]_-{L(m)} & L(D)}
\]
in $\C$. It is easy to check that $\eta_D$ is an isomorphism; hence $\eta_M$ is both a monomorphism and a regular epimorphism, so that it is an isomorphism. It follows that~$M$ is an object of $D$.
\end{proof}

\begin{exercise}
	Prove that if $\D$ is a Birkhoff subcategory of $\C$ and $\E$ is a subcategory of $\D$, then $\E$ is a Birkhoff subcategory of $\D$ iff it is a Birkhoff subcategory of $\C$.
\end{exercise}

\begin{theorem}[Birkhoff]
A variety of non-associative $\K$-algebras $\V$ is the same thing as a Birkhoff subcategory of $\Alg_\K$. 
\end{theorem}
\begin{proof}
Given a variety on non-associative $\K$-algebras $\V$, it is reflective by Proposition~\ref{Prop:Reflective} and closed under quotients by Proposition~\ref{Prop:Cokernel}. Closure under subobjects is a consequence of Lemma~\ref{Lem:RegepiReflection} and the first step in the proof of Proposition~\ref{Prop:Reflective}.

For a proof of the converse, let $L\colon {\Alg_\K\to \V}$ denote the left adjoint to the inclusion functor $\V\to \Alg_\K$. Fix a countable set of variables $X=\{x_1,x_2,\dots,x_n,\dots\}$ and take the free $\K$-algebra $\Free{X}$. Then the induced morphism
\[
\eta_{\Free{X}}\colon {\Free{X}\to L(\Free{X})}
\]
is a regular epimorphism by Lemma~\ref{Lem:RegepiReflection}. Taking the kernel of $\eta_{\Free{X}}$, we find a short exact sequence
\[
\xymatrix{0 \ar[r] & I \ar[r] & \Free{X} \ar[r]^-{\eta_{\Free{X}}} & L(\Free{X}) \ar[r] & 0\text{.}}
\] 
We shall prove that the set of polynomials $I\subseteq \Free{X}$ determines~$\V$. That is to say, a non-associative $\K$-algebra $A$ is in $\V$ if and only if it satisfies the identities in $I$. 

First suppose that $A$ is in $\V$. Let $\psi(x_1,\dots,x_n)$ be a polynomial in $I$, and consider elements $a_1$, \dots, $a_n$ of $A$. The morphism of $\K$-algebras $z\colon \Free{X}\to A$ which sends $x_i$ to $a_i$ if $0\leq i\leq n$ and to $0$ otherwise necessarily factors through $L(\Free{X})$ as $z=z'\circ \eta_{\Free{X}}$, since $A$ is in $\V$. Hence $\psi(a_1,\dots, a_n)=z(\psi )=z'(\eta_{\Free{X}}(\psi))=0$, because $\psi$ is in the kernel of $\eta_{\Free{X}}$.

Conversely, suppose that $A$ is a $\K$-algebra which satisfies the identities in $I$. We consider the surjective morphism $\varepsilon_A\colon \Free{A}\to A$ which sends $a\in A$ to itself, as well as the short exact sequence 
\[
\xymatrix{0 \ar[r] & J(A) \ar[r]^-{\kappa_A} & \Free{A} \ar[r]^-{\eta_{\Free{A}}} & L(\Free{A}) \ar[r] & 0\text{.}}
\] 
We prove that $\varepsilon_A\circ \kappa_A=0$; then $\varepsilon_A$ factors over $\eta_{\Free{A}}$ as a regular epimorphism $\varepsilon'_A\colon{L(\Free{A})\to A}$ such that $\varepsilon'_A\circ\eta_{\Free{A}}=\varepsilon_A$. It follows that $A$ is a quotient of~$L(\Free{A})$, hence an object of $\V$.

Consider an element of $J(A)$: it is a polynomial $\psi(a_1,\dots,a_n)$ in $n$ variables $a_1$, \dots, $a_n\in A$. Let $Y\subseteq X$ be the subset $\{x_1,\dots,x_n\}$ of $X$ and consider the morphism of $\K$-algebras $y\colon \Free{Y}\to \Free{A}$ which sends $x_i$ to $a_i$ for $0\leq i\leq n$. In particular, it sends $\psi=\psi(x_1,\dots,x_n)$ to $\psi(a_1,\dots,a_n)$. Note that it is a split monomorphism, because any functor preserves split monomorphisms; a splitting may be defined which sends $a_i$ to $x_i$ and all other $a\in A$ to $0$. On the other hand, the inclusion of $Y$ into~$X$ induces a split monomorphism $\Free{Y}\to \Free{X}$. Since $L$ is a functor, we find the following two vertical split monomorphisms of short exact sequences in $\Alg_\K$.
\[
\xymatrix{0 \ar[r] & J(A) \ar[r]^-{\kappa_A} & \Free{A} \ar[r]^-{\eta_{\Free{A}}} & L(\Free{A}) \ar[r] & 0\\
0 \ar[r] & K \ar[u] \ar[d] \ar[r] & \Free{Y} \ar[u]^-y \ar[d] \ar[r]_-{\eta_{\Free{Y}}} & L(\Free{Y}) \ar[u]_-{L(y)} \ar[d] \ar[r] & 0\\
0 \ar[r] & I \ar[r] & \Free{X} \ar[r]_-{\eta_{\Free{X}}} & L(\Free{X}) \ar[r] & 0}
\] 
We have that $\psi(a_1,\dots,a_n)=y(\psi)$, so
\[
0=\eta_{\Free A}(\psi(a_1,\dots,a_n))=L(y)(\eta_{\Free Y}(\psi)).
\]
Since $L(y)$ is an injection, $\psi$ is an element of $K$. Hence it is also in $I$. Taking $z\colon \Free{X}\to A$ as above and $z'\colon \Free{X}\to \Free{A}$ the morphism defined by the same rules, we have
\[
\varepsilon_A(\kappa_A(\psi(a_1,\dots,a_n)))=\varepsilon_A(z'(\psi))=z(\psi)=\psi(a_1,\dots,a_n)= 0\text{,}
\]
because $\psi$ is in $I$ and $I(A)=0$. (Note that the $\psi(a_1,\dots,a_n)$ on the left is a polynomial in the variables $a_1$, \dots, $a_n$, while the $\psi(a_1,\dots,a_n)$ on the right is an element of $A$: the evaluation of this polynomial.) It follows that $A$ is in $\V$.
\end{proof}

\begin{example}
The trivial variety is determined by $I=\Free{X}$, while $\Alg_\K$ is determined by $I=0$. 

Any variety of non-associative $\K$-algebras contains the variety $\AbAlg_\K$ of abelian $\K$-algebras (Example~\ref{Ex:Abelian}), which is determined for instance by the set $\{x_1x_2\}$. Note that the set $\{x_3x_4\}$ determines the same variety of algebras, even though the ideals $I$ and $J$ generated by these two sets are different. On the other hand, the kernel of the unit $\eta_{\Free{X}}\colon {\Free{X}\to L(\Free{X})}$, where $L\colon{\Alg_\K\to \AbAlg_\K}$ is left adjoint to the inclusion, is equal to both $I(\Free{X})$ and~$J(\Free{X})$. In other words, they generate the same $T$-ideal (Definition~\ref{Def:T-ideal}).
\end{example}

\section{Homogeneous identities}\label{S12}
Recall from Definition~\ref{Def:Polynomial} that the \defn{type} of a monomial $\varphi(x_{1},\dots,x_{n})$ is the element $(k_{1},\dots,k_{n})\in \N^{n}$ where $k_{i}$ is the degree of $x_{i}$ in $\varphi(x_{1},\dots,x_{n})$. So for each of the variables $x_{1}$,\dots, $x_{n}$, it keeps track of the number of times this variable occurs in the monomial $\varphi(x_{1},\dots,x_{n})$. Then a polynomial is said to be \defn{homogeneous} if its monomials are all of the same type, and any polynomial may thus be written as a sum of homogeneous polynomials, which are called its \defn{homogeneous components}.

We shall now prove Theorem~\ref{Shestakov} which says that, over an infinite field $\K$ (in particular, over any field of characteristic zero), when a polynomial is an identity of a variety of $\K$-algebras $\V$, then its homogeneous components are also identities of~$\V$. So for instance, the singleton set $\{x(yz)-(xy)z+xy-yx\}$ already determines the variety of associative commutative algebras. As we shall see in the next section, this result has some strong categorical-algebraic consequences.

Let $\psi(x_{1},\dots,x_{n})$ be an identity of a variety of $\K$-algebras $\V$. Write $\psi=\phi_0+\phi_1+\cdots +\phi_k$ where $\phi_i$ is the sum of all monomials in $\psi$ which are of degree $i$ in $x_1$. Now consider $k+1$ distinct elements $\alpha_1$, \dots, $\alpha_{k+1}$ of the (infinite) field~$\K$. Then the Vandermonde determinant 
\[
d=\begin{vmatrix}
	1 & \alpha_1 & \alpha_1^2 & \cdots & \alpha_1^k\\
	1 & \alpha_2 & \alpha_2^2 & \cdots & \alpha_2^k\\
	\vdots & \vdots & \vdots & \ddots & \vdots \\
	1 & \alpha_k & \alpha_k^2 & \cdots & \alpha_k^k\\
	1 & \alpha_{k+1} & \alpha_{k+1}^2 & \cdots & \alpha_{k+1}^k
\end{vmatrix}
=
\prod_{1\leq i<j\leq k+1}(\alpha_j-\alpha_i)
\]
is non-zero. Let $a_1$, \dots, $a_n$ be elements of an algebra $A$ of $\V$. Write $\phi_i(a)= \phi_i(a_{1},\dots,a_{n})$. Then for all $j\in \{1,\dots,k+1\}$ we have 
\begin{align*}
&\phi_0(a)+\alpha_j\phi_1(a)+\cdots+\alpha_j^k\phi_k(a)\\
&=\phi_0(a_{1},\dots,a_{n})+\alpha_j\phi_1(a_{1},\dots,a_{n})+\cdots+\alpha_j^k\phi_k(a_{1},\dots,a_{n})\\
&=\phi_0(\alpha_ja_{1},\dots,a_{n})+\phi_1(\alpha_ja_{1},\dots,a_{n})+\cdots+\phi_k(\alpha_ja_{1},\dots,a_{n})\\
&=\psi(\alpha_ja_1,a_2,\dots,a_n)=0
\end{align*}
where the first equality holds by definition of the $\phi_i(a)$, the second because $\phi_i$ is of degree $i$ in $x_1$, and the third since $\psi=\phi_0+\phi_1+\cdots +\phi_k$. So, in other words,
\[
\begin{ppmatrix}
	0 \\ 0 \\ \vdots \\ 0
\end{ppmatrix}
=\begin{ppmatrix}
	1 & \alpha_1 & \alpha_1^2 & \cdots & \alpha_1^k\\
	1 & \alpha_2 & \alpha_2^2 & \cdots & \alpha_2^k\\
	\vdots & \vdots & \vdots & \ddots & \vdots \\
	1 & \alpha_k & \alpha_k^2 & \cdots & \alpha_k^k\\
	1 & \alpha_{k+1} & \alpha_{k+1}^2 & \cdots & \alpha_{k+1}^k
\end{ppmatrix}
\cdot
\begin{ppmatrix}
	\phi_0(a)\\
	\phi_1(a)\\
	\vdots\\
	\phi_k(a)
\end{ppmatrix}
\]
where the matrix in the middle is invertible because $d\neq 0$. It follows that 
\[
\phi_0(a)=\phi_1(a)=\cdots=\phi_k(a)=0\text{.}
\]
Since $A$ was an arbitrary $\V$-algebra and $a_1$, \dots, $a_n$ were arbitrary elements of $A$, each $\phi_i$ is an identity of $\V$. Note that the monomials in $\phi_i$ are all of the same degree in $x_1$. So repeating this process for the variables $x_2$, \dots, $x_n$, in the end we find that the homogeneous components of $\psi$ are identities of $\V$. Thus we proved:

\begin{theorem}[Zhevlakov--Slin'ko--Shestakov--Shirshov]\label{Shestakov}
In a variety of algebras $\V$ over an infinite field $\K$, if $\psi(x_{1},\dots,x_{n})$ is an identity of $\V$, then each of its homogeneous components $\phi(x_{i_{1}},\dots, x_{i_{m}})$ is again an identity of $\V$.
\end{theorem}

\section{Some recent results}

For certain applications (in Homological Algebra, for example) the axioms of semi-abelian categories are too weak. With the aim of including such applications in the theory, over the last 15 years or so, a whole tree of interdependent additional conditions has been investigated. As a rule, such a condition strengthens the context so that it becomes closer to the abelian setting, while at the same time excluding certain examples.

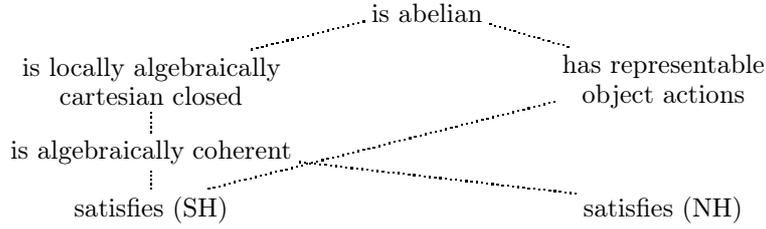
\begin{figure}
\[
\xymatrix@R=1.5ex{& \txt{is abelian} \ar@{.}[ld] \ar@{.}[rd] \\
\txt{is locally algebraically\\
cartesian closed} \ar@{.}[d] && \txt{has representable\\ object actions} \ar@{.}[lldd]\\
\txt{is algebraically coherent} \ar@{.}[d] \ar@{.}[rrd]\\
\txt{satisfies (SH)} && \txt{satisfies (NH)}}
\]
\caption{Some categorical-algebraic conditions, with currently known implications between them: a semi-abelian category\dots}\label{Figure Tree}
\end{figure}

(An instance of this process, unfortunately not within the scope of this text, is the description of the derived functors of the abelianisation functor, which in its simplest form is only valid when an additional condition holds that excludes the semi-abelian category of loops.)

See Figure~\ref{Figure Tree} for an overview of some of the conditions in this tree. Two such additional conditions turn out to be particularly relevant for us in the present context: \emph{algebraic coherence}~\cite{acc} and \emph{local algebraic cartesian closedness}~\cite{Gray2012}. There is no space here to explain what these conditions are useful for, how they were discovered, or what their consequences (SH) and (NH) mean. On the other hand, we can briefly sketch the unexpected interpretation they gain in the current setting of non-associative algebras: the former amounts to a weak associativity rule, while the latter gives \emph{a categorical characterisation of the concept of a (quasi-)Lie algebra}. 

The following is not the original definition, but it is suitable for us:

\begin{definition}
	Given objects $B$ and $X$ in a semi-abelian category $\C$, take their coproduct and then the kernel of the induced split epimorphism $\left\lgroup1_B\;0\right\rgroup\colon B+X\to B$ in order to obtain the short exact sequence
\[
\xymatrix{0 \ar[r] & B\flat X \ar[r] & B+X \ar[r]^-{\left\lgroup1_B\;0\right\rgroup} & B \ar[r] & 0\text{.}}
\] 
Fixing $B$, this process determines a functor $B\flat(-)\colon {\C\to \C}$. (We shall not explore this aspect here, but the functor $B\flat(-)$ occurs in the definition of an \emph{internal $B$-action in $\C$}: it is part of the monad whose algebras are the internal actions.) For any two objects $X$ and $Y$, we have a canonical comparison morphism
\[
\left\lgroup B\flat \iota_X \; B\flat \iota_Y\right\rgroup\colon B\flat X+ B\flat Y\to B\flat (X+Y)\text{.}
\]
The category $\C$ is called \defn{algebraically coherent} when for all $B$, $X$, $Y$ in $\C$, the morphism $\left\lgroup B\flat \iota_X \; B\flat \iota_Y\right\rgroup$ is a regular epimorphism; $\C$ is said to be \defn{locally algebraically cartesian closed (LACC)} when each $\left\lgroup B\flat \iota_X \; B\flat \iota_Y\right\rgroup$ is an isomorphism.
\end{definition}

We have the following two results, of which we shall sketch part of the proofs:

\begin{theorem}\cite{GM-VdL2}\label{Theorem AC iff Orzech}
Let $\K$ be an infinite field. If $\V$ is a variety of non-associative $\K$-algebras, then $\V$ is algebraically coherent iff there exist $\lambda_{1}$, \dots, $\lambda_{16}\in \K$ such that the equations
\begin{align*}
z(xy)=
\lambda_{1}y(zx)&+\lambda_{2}x(yz)+
\lambda_{3}y(xz)+\lambda_{4}x(zy)\\
&+\lambda_{5}(zx)y+\lambda_{6}(yz)x+
\lambda_{7}(xz)y+\lambda_{8}(zy)x
\end{align*}
and 
\begin{align*}
	(xy)z=
	\lambda_{9}y(zx)&+\lambda_{10}x(yz)+
	\lambda_{11}y(xz)+\lambda_{12}x(zy)\\
	&+\lambda_{13}(zx)y+\lambda_{14}(yz)x+
	\lambda_{15}(xz)y+\lambda_{16}(zy)x
\end{align*}
hold in $\V$.
\end{theorem}

\begin{exercise}
It follows easily that this is equivalent to $\V$ being a \defn{$2$-variety} in the sense of~\cite{Zwier}: for any ideal $I$ of an algebra $A$, the subalgebra $I^2$ of $A$ is again an ideal. As an immediate consequence, it may now be seen that in the context of varieties of non-associative algebras over an infinite field, algebraic coherence is equivalent to \emph{normality of Higgins commutators} in the sense of \cite{CGrayVdL1}---the condition (NH) in Figure~\ref{Figure Tree}.
\end{exercise}

\begin{exercise}
These conditions are also equivalent to $\V$ being an \emph{Orzech category of interest}~\cite{Orzech}.
\end{exercise}

The implication $\Rightarrow$ of Theorem~\ref{Theorem AC iff Orzech} becomes a straightforward consequence of Theorem~\ref{Shestakov} once we have a sufficiently explicit interpretation of the objects $B\flat X$. Let $B$, $X$ and $Y$ be free $\K$-algebras with a single generator $b$, $x$ and~$y$, respectively. The ideal $B\flat X$ of $B+X$ is generated by monomials of the form $\psi(b,x)$ in which $x$ occurs at least once. If now $B$, $X$ and~$Y$ are free $\V$-algebras, and $B\flat X$ is computed in $\V$, then the only difference is that we need to take classes of such polynomials, modulo the identities of~$\V$.

Algebraic coherence now means that the element $b(xy)$ of $B\flat (X+Y)$ may be obtained as the image of some polynomial $\psi(b_{1},x,b_{2},y)$ in $B\flat X+ B\flat Y$ through the function $\left\lgroup B\flat \iota_X \; B\flat \iota_Y\right\rgroup$. Note that this polynomial cannot contain any monomials obtained as a product of a $b_{i}$ with $xy$ or $yx$. This allows us to write, in the sum $B+X+Y$, the element $b(xy)$ as 
\begin{align*}
\lambda_{1}y(bx)&+\lambda_{2}x(yb)+
\lambda_{3}y(xb)+\lambda_{4}x(by)\\
&+\lambda_{5}(bx)y+\lambda_{6}(yb)x+
\lambda_{7}(xb)y+\lambda_{8}(by)x\\
&+\nu\phi(b,x,y)
\end{align*}
for some $\lambda_{1}$, \dots, $\lambda_{8}$, $\nu\in \K$, where $\phi(b,x,y)$ is the part of the polynomial in $b$, $x$ and~$y$ which is not in the homogeneous component of $b(xy)$. Since $B+X+Y$ is the free $\V$-algebra on three generators $b$, $x$ and $y$, from Theorem~\ref{Shestakov} we deduce that the first equation in Theorem~\ref{Theorem AC iff Orzech} is again an identity in $\V$. Analogously, for~$(xy)b$ we deduce the second equation.

\begin{theorem}\cite{GM-VdL3}
Let $\K$ be an infinite field. Let $\V$ be a non-abelian \LACC\ variety of $\K$-algebras. Then 
\begin{enumerate}
\item $\V = \Liek = \qLiek$ when $\kar(\K)\neq 2$;
\item $\V = \Liek$ or $\V = \qLiek$ when $\kar(\K)= 2$.
\end{enumerate}
\end{theorem}

Let us first consider the special case where $xy=-yx$ in $\V$. Then we can reduce the first equation of Theorem~\ref{Theorem AC iff Orzech} to $z(xy) = \lambda y(zx) + \mu x(yz)$, for some $\lambda$, $\mu \in \K$. Considering now $y = z$, and then $x = z$, we deduce that either $\lambda = \mu = -1$, or $z(zx) = 0$ is an identity of $\V$. The first case is exactly the Jacobi identity. In the second case, we see that $0 = (x + y)((x+y)b) = x(yb) + y(xb)$. Therefore, the comparison map sends $x(yb_2) - y(xb_1) \in B\flat X + B\flat Y$ to zero in $B \flat (X + Y)$, which via \LACC\ and Theorem~\ref{Shestakov} implies that $x(yz)$ is an identity of $\V$. Variations on these ideas allow us to prove that when $\V$ is \LACC\ and $x(yz)=0$ in $\V$, the variety is necessarily abelian; hence if $\V$ is non-abelian, then the Jacobi identity must hold.

In general, not assuming anti-commutativity, this result is much harder to prove. Our current strategy involves a proof by computer, which shows that a certain system of polynomial equations is inconsistent.

In this context, many basic questions currently remain unanswered. In view of the above result, we may ask questions such as, for instance: \emph{What is associativity?} Can it be captured in terms of a categorical-algebraic condition?

\section{Bibliography}

The aim of these notes is to provide a gentle introduction to two subjects at the same time, semi-abelian categories and non-associative algebras, focusing on the connections between them. ``Gentle'' here is supposed to mean that little prior knowledge of either subject is required to understand most of the results. There is nothing new here: the content is essentially a rearrangement of known results, often not in their most general form, with easier proofs preferred over more efficient ones. The notes are as self-contained as possible, which necessarily means that they are incomplete in many different ways. Here follows an overview of some works which go beyond what is presented in this text.

\begin{itemize}
\item Non-associative algebras: \cite{Bahturin} and \cite{Shestakov}
\item Basic category theory: \cite{Borceux:Cats1} and \cite{MacLane}
\item Protomodular, regular, exact, semi-abelian categories: \cite{Borceux-Semiab}, \cite{Borceux-Bourn}, \cite{Bourn-Gran-CategoricalFoundations} and \cite{Janelidze-Marki-Tholen}
\item (Co)homological applications: \cite{EGVdL} and \cite{RVdL2}
\item Birkhoff subcategories: \cite{Janelidze-Kelly}
\item Recent results: \cite{acc}, \cite{GM-VdL2}, \cite{GM-VdL3} and \cite{Gray2012}
\end{itemize}

\section*{Acknowledgements}
Many thanks to Xabi García-Martínez, who introduced me to non-associative algebras and made me want to understand them better. Thanks to the participants of the \emph{Summer School in Algebra and Topology}, and to François Renaud, Corentin Vienne and the anonymous referees for their feedback on the text. 


\providecommand{\noopsort}[1]{}
\providecommand{\bysame}{\leavevmode\hbox to3em{\hrulefill}\thinspace}
\providecommand{\MR}{\relax\ifhmode\unskip\space\fi MR }
\providecommand{\MRhref}[2]{%
  \href{http://www.ams.org/mathscinet-getitem?mr=#1}{#2}
}
\providecommand{\href}[2]{#2}

\end{document}